\documentclass[11pt,leqno]{amsart}

\usepackage{graphicx}
\usepackage{epstopdf,epsfig}
\usepackage[colorlinks=false]{hyperref}
\usepackage{csquotes}

\usepackage{todonotes}
\usepackage{dsfont}

\usepackage{xcolor}

\usepackage{amsmath,amsthm,amsfonts,amssymb,latexsym,amscd,enumerate,mathtools}

\usepackage{xy}
\xyoption{all}

\theoremstyle{remark}

\theoremstyle{plain}

\newcounter{theoremintro} 
\newtheorem{introdefinition}[theoremintro]{Definition}
\newtheorem{introtheorem}[theoremintro]{Theorem}
\newtheorem{introcorollary}[theoremintro]{Corollary}

\newtheorem*{definition*}{Definition} 
\newtheorem*{theorem*}{Theorem} 
\newtheorem*{lemma*}{Lemma}
\newtheorem*{corollary*}{Corollary} 

\newtheorem{claim}{Claim}
\newtheorem{theorem}[subsection]{Theorem} 
\newtheorem{lemma}[subsection]{Lemma}
\newtheorem{corollary}[subsection]{Corollary}
\newtheorem{proposition}[subsection]{Proposition}

\theoremstyle{definition}
\newtheorem{definition}[subsection]{Definition}

\theoremstyle{remark}
\newtheorem{example}[subsection]{Example}
\newtheorem{remark}[subsection]{Remark}

\numberwithin{equation}{section}

\newcommand{\SL}{\mathrm{SL}}

\newcommand{\FIN}{\mathrm{FIN}}

\newcommand{\N}{\mathbb{N}}
\newcommand{\Z}{\mathbb{Z}} 
\newcommand{\Q}{\mathbb{Q}}

\newcommand{\C}{\mathbb{C}} 

\newcommand{\Aut}{\mathrm{Aut}}

\newcommand{\Ab}{\mathrm{Ab}}

\newcommand{\Mod}{\mathrm{Mod}\text{-}}
\newcommand{\Hom}{\mathrm{Hom}}

\DeclareMathOperator{\End}{End}
\DeclareMathOperator{\id}{id}

\DeclareMathOperator{\Ind}{Ind}
\DeclareMathOperator{\Idem}{Idem}
\newcommand{\Sw}[1]{\mathrm{Sw}^{fr}(#1;\Z)}

\begin{document} 
	\title[Bernoulli shifts and algebraic $K$-theory]{Bernoulli shifts on additive categories and algebraic $K$-theory of wreath products}
	
		\author[J. Kranz]{Julian Kranz}
	\address{J.K.: Chair of Data Science,		
		Institut für Wirtschaftsinformatik,		
		Leonardo Campus 3, 48149 M\"unster, Germany.}
	\email{julian.kranz@uni-muenster.de}
	\urladdr{https://sites.google.com/view/juliankranz/}

	\author[S. Nishikawa]{Shintaro Nishikawa} 
	\address{S.N.: School of Mathematical Sciences,		
		University of Southampton, 		
		University Road,		
		Southampton,		
		SO17 1BJ,		
		United Kingdom.} 
	\email{s.nishikawa@soton.ac.uk}
	\urladdr{https://sites.google.com/view/snishikawa/}
	
	\thanks{J.K. was partially supported by the Engineering and Physical Sciences Research Council [Grant Ref: EP/X026647/1]. 
		Both authors were partially funded by the Deutsche Forschungsgemeinschaft (DFG, German Research Foundation) - Project-ID 427320536 - SFB 1442, as well as by Germany's Excellence Strategy EXC
		2044 390685587, Mathematics Münster: Dynamics-Geometry-Structure.}

	%    General info
	\date{\today}
	
	
	\maketitle
	
	\begin{abstract}
		We develop general methods to compute the algebraic $K$-theory of crossed products by Bernoulli shifts on additive categories. From this we obtain a $K$-theory formula for regular group rings associated to wreath products of finite groups by groups satisfying the Farrell--Jones conjecture. 
	\end{abstract}

	\tableofcontents
	
	\section{Introduction}
	
	The \emph{$K$-theoretic Farrell--Jones conjecture} \cite{FJoriginal} predicts that for any discrete group $G$ and any ring $R$, the \emph{assembly map} 
		\begin{equation}\label{eq:FJ}
			\mu\colon H^G_*(\mathcal E_{\mathrm{VCyc}}G,\mathbb K_R)\to K_*(R[G])
		\end{equation}
	is an isomorphism. The right-hand side denotes the algebraic $K$-theory groups of the group ring $R[G]$ and is of great interest for many problems in algebraic topology.	
	The conjecture implies several other famous conjectures such as the Bass conjectures and Kaplansky's idempotent conjecture. 
	It has been verified for a large class of groups including all hyperbolic groups \cite{Bartelshyperbolic}, CAT(0)-groups \cite{cat0}, and mapping class groups \cite{FJmappingclass}. %, virtually solvable groups \cite{FJvirtuallysolvable} and many linear groups \cite{FJGLNZ,FJvirtuallyconnected}.
	All these groups satisfy a more general \emph{$K$-theoretic Farrell--Jones conjecture with coefficients} \cite[Conjecture 3.2]{bartels2007coefficients}, which says that for any  additive category $\mathcal A$ with $G$-action, the assembly map
		\begin{equation}\label{eq:FJcoeff}
			\mu\colon H^G_*(\mathcal E_{\mathrm{VCyc}}G,\mathbb K_{\mathcal{A}})\to K_*(\mathcal A\rtimes G)
		\end{equation}
is an isomorphism. The right-hand side denotes the algebraic $K$-theory groups $K_*(\mathcal A\rtimes G)$ (see Section \ref{sec:Rlinear} for the definition of $\mathcal A\rtimes G$) where $K_i(-)=\pi_i(\bold{K}^{\infty}(-))$ $(i\in\Z)$ for the non-connective algebraic $K$-theory functor $\bold{K}^\infty$ for additive categories \cite[Definition 4.1]{Nonconnective}.
	%The conjecture with coefficients reduces to the classical one when $\mathcal A= \mathcal F^p(R)$ is the category of fintiely generated projective $R$-modules with the trivial $G$-action.  
	This paper is devoted to computing the algebraic $K$-theory of certain (twisted) group rings, \emph{using} the Farrell--Jones conjecture as an assumption. 
	We refer to \cite{Lueck2024, Reich2018} for surveys on the Farrell--Jones conjecture and to \cite{BCH,BCsurvey} for the closely related \emph{Baum--Connes conjecture} about topological $K$-theory of group $C^*$-algebras.

	The left-hand side of \eqref{eq:FJ} is given by equivariant homology groups of the classifying space $\mathcal E_{\mathrm{Vcyc}}G$ for the family of virtually cyclic subgroups with respect to an equivariant homology theory satisfying $H^G_*(G/H,\mathbb K_R)\cong K_*(R[H])$ for all subgroups $H\subseteq G$. 
	These homology groups are still much more difficult to compute than just computing $K_*(R[H])$ for all virtually cyclic subgroups $H\subseteq G$ since they also depend on the equivariant CW-structure of $\mathcal E_{\mathrm{Vcyc}}G$.	
	The functoriality properties of the assembly map however imply a very useful principle	
	which is analogous to the Going-Down-principle \cite{CEO} in the context of the Baum--Connes conjecture	\footnote{Although the Going--Down-principle is a triviality in the Davis--L\"uck framework \cite{davislueck}, it is a non-trivial theorem in the original framework by Baum--Connes--Higson \cite{CEO}. In fact, it is the key ingredient in identifying these two frameworks.}:
	\begin{lemma*}[Going-Down-principle]
	Assume that $G$ satisfies the Farrell--Jones conjecture with coefficients. 
	If $F\colon \mathcal A\to \mathcal B$ is an equivariant additive functor between additive categories with $G$-action such that 
	\begin{equation}\label{eq:localequivalence}
		K_*(F\rtimes H)\colon K_*(\mathcal A\rtimes H)\xrightarrow{\cong} K_*(\mathcal B\rtimes H)
	\end{equation}
	is an isomorphism for every virtually cyclic subgroup $H\subseteq G$, then
	\[K_*(F\rtimes G)\colon K_*(\mathcal A\rtimes G)\xrightarrow{\cong} K_*(\mathcal B\rtimes G)\]
	is an isomorphism. 
	\end{lemma*}
	When assuming that the categories $\mathcal A$ and $\mathcal B$ are sufficiently regular, the above situation even reduces to all finite subgroups $H\subseteq G$ instead of all virtually cyclic subgroups (see Section \ref{sec-wreath}).
	
	We use this principle to compute the algebraic $K$-theory of categories of the form $\mathcal A^{\otimes Z}\rtimes G$ where $\mathcal A$ is an additive category and $\mathcal A^{\otimes Z}$ is the infinite tensor product taken over a $G$-set $Z$. We moreover treat the more general situation where $\mathcal A$ is linear over a commutative base ring $R$ and the tensor products are taken relative to $R$ (see Section \ref{sec:Rlinear} for the definitions). 
	%Our main motivation is to compute the $K$-theory of group rings of wreath products $H\wr G$, using the canonical isomorphism $R[H\wr G]\cong R[H]^{\otimes G}\rtimes G$. 
	Our main strategy is to replace $\mathcal A$ by a more tractable category $\mathcal B$ which is $K$-theoretically equivalent to $\mathcal A$ in the following sense: 
	\begin{introdefinition}[Definition \ref{defn:virtualeq}]
		A \emph{virtual equivalence} between $R$-linear categories is an $R$-linear functor $F\colon \mathcal A\to \mathcal B$ such that there exist $R$-linear functors $Q_\pm\colon \mathcal B\to \mathcal A$ and natural equivalences 
			\[FQ_+\simeq \id_\mathcal B\oplus FQ_-,\quad Q_+F\simeq \id_\mathcal A \oplus Q_-F.\]
	\end{introdefinition}
	Virtual equivalences induce isomorphisms at the level of $K$-theory: in the notation above, the inverse of $K_*(F)$ is given by $K_*(Q_+)-K_*(Q_-)$. Note that our definition can be viewed as an additive analogue of \emph{universal $K$-equivalences} from \cite[Definition 4.4]{Barwick2022}\footnote{A subtle difference is that we don't allow for stabilization by additional functors $S,T$ such that $FQ_+\oplus S\simeq \id_\mathcal B\oplus FQ_-\oplus S$ and $ Q_+F\oplus T\simeq \id_\mathcal A \oplus Q_-F\oplus T$ as in \cite[Definition 4.4]{Barwick2022}. We decided on this simplification in favor of shorter proofs rather than conceptual reasons.}.
	
	Our definition allows us to prove an analogue of Izumi's theorem \cite[Theorem 2.1]{Izumi} which says that the Bernoulli shifts of a finite group $G$ on two $KK$-equivalent $C^*$-algebras are $KK^G$-equivalent in the sense of \cite{Kasparov}. Our substitute for the $C^*$-algebraic $KK$-equivalences are precisely the virtual equivalences:
	\begin{introtheorem}[Theorem \ref{lem:virtualEqFiniteGroup}]\label{introthm:Izumi}
		Let $F\colon \mathcal A\to \mathcal B$ be an $R$-linear virtual equivalence, let $G$ be a finite group and let $Z$ be a finite $G$-set. Then 
			\[F^{\otimes Z}\rtimes G\colon \mathcal A^{\otimes Z}\rtimes G\to \mathcal B^{\otimes Z}\rtimes G\]
		is an $R$-linear virtual equivalence as well.
	\end{introtheorem}
	Under sufficient regularity assumptions on the underlying categories, the Going-Down principle allows us to upgrade the above theorem from finite groups to groups satisfying the Farrell--Jones conjecture. We apply this strategy to group rings associated to wreath products by realizing them as crossed product rings $R[H\wr G]\cong R[H]^{\otimes G}\rtimes G$ associated to the natural shift action:
	\begin{introtheorem}[Theorem \ref{thm:wreathproducts}]
		Let $G$ be a group satisfying the $K$-theoretic Farrell--Jones conjecture with coefficients and let $H$ be a finite group. Let $R$ be a regular commutative ring such that the orders of $H$ and of every finite subgroup of $G$ are invertible in $R$. 
		Denote by $\mathrm{FIN}(G)$ the set of finite subsets of $G$ and by $I_R[H]$ the augmentation ideal of $R[H]$, which is an $R$-algebra with unit by assumption. Then we have an isomorphism
		\[K_*(R[H\wr G])\cong \bigoplus_{[F]\in G\backslash\mathrm{FIN}(G)}K_*(I_R[H]^{\otimes F}\rtimes G_F).\]
		Here $G$ acts on $\mathrm{FIN}(G)$ by left translation and $G_F$ denotes the stabilizer of $F\in \mathrm{FIN}(G)$. 
	\end{introtheorem}
	The above result is an algebraic $K$-theory analogue of the results in \cite{XinLi}, see also \cite{CEKN}. 
	The commutativity assumption on $R$ can be dropped (see Remark \ref{rem:wreathproducts}).
	The formula becomes much more concrete for algebraically closed fields of characteristic zero:
	
	\begin{introcorollary}[Corollary \ref{cor:complexgroupring}]
	Let $G$ be a group satisfying the $K$-theoretic Farrell--Jones conjecture with coefficients, let $H$ be a finite group, and let $R$ be an algebraically closed field of characteristic zero.
	Then we have an isomorphism 
	\begin{align*}
		K_*(R[H\wr G])
		&\cong \bigoplus_{[F]\in G\backslash \FIN(G)} \bigoplus_{[S] \in G_F\backslash (\{1, \ldots, n\}^F) } K_*(R[G_S])\\
		&\cong  K_*(R[G]) \oplus \bigoplus_{[C]\in \mathcal C}\,\, \bigoplus_{[X] \in N_C \backslash F(C) }\,\,  \bigoplus_{[S] \in C \backslash \{1, \ldots, n\}^{C\cdot X} } K_\ast(R[C_S]).			
	\end{align*}
	\end{introcorollary}
	Here, $n$ denotes the number of non-trivial conjugacy classes of $H$, $\mathcal C$ denotes the set of all conjugacy classes of finite subgroups of $G$, $F(C)$  the nonempty finite subsets of $C\backslash G$ which are not of the form $\pi^{-1}(Y)$ for a finite subgroup $D\subseteq G$ with $C\subsetneq D$ and $Y\subseteq D\backslash G$ where $\pi\colon C\backslash G\to D\backslash G$ denotes the projection, $N_C=\{g\in G: gCg^{-1}=C\}$ the normalizer of $C$ in $G$, and $C_S=G_S\cap C$ the stabilizer of $S$ in $C$.
	
	As another application of the Going-Down principle, we compute the algebraic $K$-theory of Bernoulli shifts on many semi-simple algebras. We do this by carefully analysing the $A(G)$-module structure of $K_*(\mathcal A\rtimes G)$ where $A(G)$ denotes the Burnside ring. These results are inspired by the $C^*$-algebraic results in \cite{KN,CEKN}.
		\begin{introtheorem}[Theorem \ref{thm-semisimple}]
		Let $G$ be a group satisfying the Farrell--Jones conjecture with coefficients. Assume that $R$ is a regular commutative ring, and that the orders of all finite subgroups of $G$ are invertible in $R$. Let $Z$ be an infinite $G$-set and let $A$ be an $R$-algebra of the form $A=M_{n_0}(R)\oplus \dotsb\oplus M_{n_k}(R)$\footnote{Note that any finite-dimensional semi-simple algebra over an algebraically closed field is of this form.}. Write $n=\mathrm{gcd}(n_0,\dotsc,n_k)$ and denote by $\mathrm{FIN}(Z)$ the set of finite subsets of $Z$ equipped with the left translation action of $G$. Then we have an isomorphism 
		\[K_*(A^{\otimes Z}\rtimes G)\cong \bigoplus_{[F]\in G\backslash \FIN(Z)}\bigoplus_{[S]\in G_F\backslash \{1,\dotsc,k\}^F}K_*(R[G_S])[1/n].\]
	\end{introtheorem}
	
	Most of the results in this paper are algebraic $K$-theory analogues of results in \cite{Izumi,XinLi,KN,CEKN}. The two main obstacles in proving these results are 
	\begin{enumerate}
		\item the necessity to deal with virtually cyclic subgroups instead of finite subgroups,
		\item the lack of Kasparov's equivariant $KK$-theory for algebraic $K$-theory.
	\end{enumerate}
	We overcome the first obstacle by imposing regularity assumptions on the underlying categories and by combining the results from \cite{JamesQuinnReich,bartels2022vanishing} to reduce the problems to finite subgroups. This reduction is the main reason why we did not extend our results to $L$-theory. We overcome the second obstacle by introducing our notion of virtual equivalences between additive categories and proving Theorem \ref{introthm:Izumi}. We moreover substitute the representation-theoretic arguments from \cite{KN} that use the representation ring $R(G)=KK^G(\C,\C)$ by combinatorial arguments using the Burnside ring $A(G)$. 
	A possible alternative approach to the second obstacle could be to set up an equivariant version of Blumberg--Gepner--Tabuada's $\infty$-category of noncommutative motives \cite{BlumbergGepnerTabuada} as a replacement of $KK^G$ and to prove Theorem \ref{introthm:Izumi} using this framework. If one is only interested in computing homotopy $K$-theory rather than algebraic $K$-theory, one may moreover use Ellis' equivariant algebraic $kk$-theory \cite{ellis} to prove such a result.
	While these approaches would certainly be more conceptual, our approach has the advantage of being elementary and of avoiding the use of heavy machinery such as stable $\infty$-categories. We hope that this choice makes our paper accessible to a broader audience. We moreover refer the reader to the recent work of Hilman \cite{Hilman} and to upcoming work of Bartels--Nikolaus for possible definitions of the $\infty$-category of $G$-equivariant noncommutative motives. %\comment{We should ask them if we can mention them like this.}
	
	\subsection*{Acknowledgements}
	We would like to thank Arthur Bartels for many helpful discussions and suggestions.
	We are moreover greatful to the anonymous referee for a very detailed and helpful report.
	
	\section{$R$-linear categories}\label{sec:Rlinear}

	Throughout this paper, let $R$ be a commutative ring with unit. All tensor products appearing in this paper will be taken over $R$ and will be denoted by $M\otimes N$ rather than $M\otimes_R N$. 
	An \emph{$R$-linear category} is a category which is enriched over $R$-modules. 
	%\comment{The terminology with $R$-linear categories etc. is just a suggestion. I am open to different terminology!}
	An \emph{$R$-additive category} is an $R$-linear category which is moreover additive. %\footnote{Note that in this case, the induced abealian group structure on morphism sets agrees with the one induced by the $R$-module structure by the Eckmann-Hilton principle.}.
	An \emph{$R$-linear functor} between $R$-linear categories is a functor which is $R$-linear on morphism sets. 
	Note that $R$-linear functors between $R$-additive categories preserve direct sums. 
	A basic example of an $R$-linear category is the ring $R$ itself considered as a category with one object. 
	A basic example of an $R$-additive category is the category $\mathcal F^f(R)$ of finitely generated free $R$-modules. 
	
	\begin{remark}
		In order to avoid set-theoretical issues, we assume the existence of sufficiently many Grothendieck universes which we call \emph{small sets, large sets, huge sets}, etc. (see \cite[\S 1.2.15]{Lurie} for this standard procedure). In this way, the collection of all small sets becomes a large set, the collection of all large sets becomes a huge set, and so on. If nothing else is said, ordinary sets are assumed to be small and ordinary categories are assumed to have large object sets and small morphism sets. This framework allows us to rigorously talk about the (huge) category of all (ordinary) categories, the (large) category of all functors between two (ordinary) categories, or (large) direct sums indexed over the objects of an (ordinary) category.
		%		
		%		Alternatively, the reader may assume that all $R$-linear categories appearing in this paper (except for the category of modules over an $R$-algebra or the category of $\mathbb Z\mathcal A$-modules over an $R$-linear category, see Section \ref{sec-wreath}) are essentially small. 
	\end{remark}
	
	For a finite set $(\mathcal A_i)_{i\in I}$ of $R$-linear categories, we define their tensor product $\bigotimes_{i\in I}\mathcal A_i$ as follows: Objects of $\bigotimes_{i\in I}\mathcal A_i$ are tuples $(A_i)_{i\in I}$ of objects $A_i\in \mathcal A_i$ and morphisms $(A_i)_{i\in I}\to (B_i)_{i\in I}$ are given by elements of  $\bigotimes_{i\in I} \Hom_{\mathcal A_i}(A_i,B_i)$. We also write $\mathcal A^{\otimes I}\coloneqq \otimes_{i\in I}\mathcal A$ for a single $R$-linear category $\mathcal A$.	We use the suggestive notation $\otimes_{i\in I}A_i\coloneqq (A_i)_{i\in I}$. The tensor product of $R$-additive categories is not necessarily additive. 
	
	Let $G$ be a group and let $\mathcal A$ be an $R$-linear category. An \emph{action} $\alpha\colon G\curvearrowright \mathcal A$ assigns to each $g\in G$ a functor $\alpha_g\colon \mathcal A\to \mathcal A$ such that $\alpha_1=\id$ and $\alpha_g\alpha_h=\alpha_{gh}$ for all $g,h\in G$. We also call $(A,\alpha)$ (or $A$ if $\alpha$ is understood) a \emph{$G$-$R$-linear category}. An $R$-linear functor $F\colon (A,\alpha)\to (B,\beta)$ is called \emph{equivariant} if we have $F\circ \alpha_g= \alpha_g \circ F$ for all $g\in G$. A natural transformation $\eta\colon F_1\Rightarrow F_2$ between equivariant functors is called \emph{equivariant} if it satisfies $\alpha_g(\eta_A)=\eta_{\alpha_g(A)}$ for all $g\in G$ and $A\in \mathcal A$. 
	In these definitions we insist on equalities rather than equivalences.
	The standard example for a $G$-$R$-linear category is the following: Let $S$ be an $R$-algebra with a left $G$-action by $R$-algebra automorphisms. We define an action $\alpha$ of $G$ on the category $\Mod S$ of right $S$-modules by mapping for each $g\in G$ any right $S$-module $M$ to the right $S$-module $\alpha_g(M)$ which is given by $M$ as a set and equipped with the multiplication $(m,s)\mapsto m\cdot g^{-1}(s)$. 
	
	We introduce an equivariant version of preferred direct sums (c.f. \cite[Section 1.3]{twistedBassHellerSwan}) which is a functorial way of adjoining direct sums. Let $\mathcal A$ be an $R$-linear category with $G$-action $\alpha$. We fix an infinite cardinal $\kappa$ and a $G$-set $\mathcal U$ such that $\mathcal U\times \mathcal U \cong \mathcal U$ and such that $\mathcal U$ contains isomorphic copies of all $G$-sets of cardinality at most $\kappa$. For instance, we can take $\kappa=|\N \times \mathcal P(G)|$ where $\mathcal P(G)$ denotes the power set of $G$ and
		\[\mathcal U\cong \bigsqcup_{H<G}\bigsqcup_{n\in \kappa} G/H,\] 
	where $H$ runs over all subgroups of $G$. Choose an invariant base point $u\in \mathcal U$ and a $G$-equivariant bijection $\tau\colon \mathcal U\times \mathcal U\to\mathcal  U$ satisfying $\tau(u,u)=u$. 
	We denote by $\mathcal A^\kappa$ the category whose objects are tuples $(A_z)_{z\in Z}$, where $Z\subseteq \mathcal U$ is a $G$-set and $A_z\in \mathcal A$ are objects. Morphisms $(A_z)_{z\in Z}\to (B_y)_{y\in Y}$ are given by column-finite matrices $(f_{z,y}\colon A_z\to B_y)_{z\in Z,y\in Y}$ and composition is given by matrix multiplication. 
	We equip $\mathcal A^\kappa$ with the induced $G$-action $\tilde \alpha$ given by $\tilde \alpha_g(A_z)_{z\in Z}\coloneqq (\alpha_g (A_{g^{-1}z}))_{z\in Z}$ for $(A_z)_{z\in Z}\in \mathcal A^\kappa$ and $g\in G$.
	For a collection $\{Z_i\}_{i\in I}$ of $G$-sets indexed over a $G$-set $I$, we define the graph $\mathcal G(\{Z_i\}_{i\in I})\coloneqq \{(z,i)\in \mathcal U\times \mathcal U\mid z\in Z_i\}$. Then $(A_{z'})_{z=\tau(z', i)\in \tau(\mathcal G(\{Z_i\}_{i\in I}))}$ is a concrete model for the direct sum $\bigoplus_{i\in I}(A_z)_{z\in Z_i}$. We also use the suggestive notation $\bigoplus_{z\in Z}A_z\coloneqq (A_z)_{z\in Z}$. 
	We denote by $\mathcal A^f\subseteq \mathcal A^\kappa$ the full subcategory of all objects $(A_z)_{z\in Z}$ where $A_z\simeq 0$ for all but finitely many $z\in Z$. Note that the inclusion $\mathcal A\to \mathcal A^f$ induced by the choice of base-point $u\in \mathcal U$ is an equivalence of categories if and only if $\mathcal A$ is additive. 

	We call $\mathcal A^f$ the \emph{additive completion} of $\mathcal A$. Our primary motivation for introducing $\mathcal{A}^f$ is its use in Section 3. There, given any $G$-set $Z$ and a $G$-equivariant family of functors $(F_z)_{z \in Z}$ from $\mathcal{A}$ to $\mathcal{B}$, we will need a naturally-defined  $G$-equivariant functor $\bigoplus_{z \in Z} F_z$ from $\mathcal{A}$ to $\mathcal{B}^f$.
		 
		 % The functor $\mathcal{A} \mapsto \mathcal{A}^f$ is left adjoint to the forgetful functor from the category of $G$-$R$-additive categories equipped with a preferred $G$-equivariant direct sum and $G$-$R$-linear functors preserving these direct sums, into the category of $G$-$R$-linear categories and $G$-$R$-linear functors. 
	
	For a $G$-$R$-linear category $(\mathcal A,\alpha)$, we define an $R$-linear category $\mathcal A\rtimes_\alpha G$ (or $\mathcal A\rtimes G$ if $\alpha$ is understood) as follows. Objects of $\mathcal A\rtimes G$ are objects of $\mathcal A$. A morphism from $A$ to $B$ in $\mathcal A\rtimes G$ consists of a finite formal sum $\sum_{g\in G} f_g g$ where each $f_g\colon \alpha_g(A)\to B, g\in G$ is a morphism in $\mathcal A$. The composition of morphisms is determined by the rule 
	\[(fg)\circ  (f'g')=(f\circ \alpha_g(f'))gg'\]
	for $g,g'\in G, f\in \mathcal A(\alpha_g(B),C)$ and $f'\in \mathcal A(\alpha_{g'}(A),B)$.
	Note that this category was denoted by $\mathcal A*_G\mathrm{pt}$ in \cite[Definition 2.1]{bartels2007coefficients}. 
	If the $G$-action on $\mathcal A$ is trivial, we also use the notation $\mathcal A[G]$ instead of $\mathcal A\rtimes G$. 
	
	\begin{remark}[see {\cite[Remark 2.3]{bartels2007coefficients}}]
		Any $G$-equivariant $R$-linear functor $F\colon \mathcal A\to \mathcal B$ between $G$-$R$-linear categories induces an $R$-linear functor 
			\begin{align*}
				F\rtimes G\colon \mathcal A\rtimes G&\to \mathcal B\rtimes G,\\
				A&\mapsto F(A),\\
				\sum_{g\in G}f_g g&\mapsto \sum_{g\in G}F(f_g)g.
			\end{align*}
		Moreover, $F\mapsto F\rtimes G$ maps equivalences of categories to equivalences of categories. More generally, if $F,F'\colon \mathcal A\to \mathcal B$ are $G$-equivariant $R$-linear functors that are equivalent via some $G$-equvariant natural transformation $\eta\colon F\Rightarrow F'$, then $F\rtimes G$ and $F'\rtimes G$ are equivalent.
	\end{remark}
	
	\begin{example}[{\cite[Example 2.6]{bartels2007coefficients}}]
		Let $S$ be an $R$-algebra with unit and let $G\curvearrowright S$ be a group action by ring automorphisms. We denote by $S\rtimes G$ the $R$-algebra of finite formal sums $\sum_{g\in G}s_g g$ with $s_g\in S$ and multiplication determined by 
		\[(sg)\cdot (s'g')\coloneqq s\cdot g(s') gg'\]
		for $s,s'\in S$ and $g,g'\in G$. \footnote{Note that $S\rtimes G$ is precisely the result of applying the functor $\rtimes G$ to the $R$-linear category with one object and $S$ as endomorphisms.}
		
		The action $G\curvearrowright S$ induces an action $\alpha\colon G \curvearrowright \mathcal F^f(S)$ of $G$ on the category finitely generated free $S$-modules by mapping a module $M$ to the module $\alpha_g(M)$ whose underlying abelian group is that of $M$ and whose $S$-module structure is given by $(m,s)\mapsto m \alpha_{g^{-1}}(s)$. 
		Then we have an equivalence 
		\[\mathcal F^f(S\rtimes G)\simeq \mathcal F^f(S)\rtimes G\]
		of $R$-additive categories.
	\end{example}

%	\begin{lemma}
%		$\mathcal A\otimes\mathcal B$ is an $R$-additive category.
%	\end{lemma}
%	\begin{proof}
%		Since $R$ is commutative, $\mathcal A\otimes \mathcal B$ is enriched over $R$-modules in an obvious way. 
%		If $0_\mathcal A,0_\mathcal B$ are zero-objects in $\mathcal A$ and $\mathcal B$ respectively, then $0_\mathcal A\otimes 0_\mathcal B$ is a zero object in $\mathcal A\otimes \mathcal B$. 
%		Since the composition in $\mathcal A\otimes \mathcal B$ is given by matrix multiplication, it is easy to see that for objects $(A_i\otimes B_i)_{i=1}^n$ and $(A_i\otimes B_i)_{i=n+1}^{n+m}$, the object $(A_i\otimes B_i)_{i=1}^{n+m}$ is a direct sum object. 
%		It follows that $\mathcal A\otimes \mathcal B$ is pre-additive. Since $\mathcal A\otimes \mathcal B$ is also enriched over $R$-modules, the Eckmann-Hilton principle implies that it is actually additive. 
%	\end{proof}
	
	\begin{example}
		Let $S, T$ be $R$-algebras with unit. Then we have equivalences of $R$-additive categories 
		\[(\mathcal F^f(S)\otimes \mathcal F^f(T))^f\simeq \mathcal F^f(S\otimes T)\]
		where $S\otimes T$ denotes the tensor product of $R$-algebras. 
	\end{example}
	
	By a \emph{based category}, we mean a category $\mathcal A$ with a preferred object $\mathds{1}_\mathcal A\in \mathcal A$. 
	A functor $F\colon \mathcal A\to \mathcal B$ between based categories is called based, if it satisfies $F(\mathds{1}_\mathcal A)= \mathds{1}_\mathcal B$. The strict identity condition is not essential and could be relaxed to isomorphism with minor adjustments to the subsequent discussion. However, we chose to retain the strict condition for convenience.
	
	\begin{definition}
		If $\mathcal A$ is a based $R$-linear category and if $Z$ is a (possibly infinite) set, we define the tensor product $\mathcal A^{\otimes Z}$, or more precisely ($\mathcal A, \mathds{1}_\mathcal A)^{\otimes Z}$, as the filtered colimit
		\[\varinjlim_{F\Subset Z}\mathcal A^{\otimes F}.\]
		Here $F$ ranges over all finite subsets of $Z$ and the connecting functors
		\[\mathcal A^{\otimes F}\to \mathcal A^{\otimes F'}\]
		are given by $A\mapsto A\otimes \mathds{1}_\mathcal A^{\otimes F'\setminus F}$ on objects and by $f\mapsto f\otimes \id_{\mathds{1}_\mathcal A}^{\otimes F'\setminus F}$ on morphisms. 
	\end{definition}
	
		The colimit is taken inside the category of $R$-linear categories with $R$-linear functors as morphisms. Note that infinite tensor products of based $R$-linear categories are functorial with respect to based $R$-linear functors.

%	The colimit is taken inside the category of $R$-linear categories with $R$-linear functors as morphisms. Note that the tensor product $\mathcal A^{\otimes Z}$ depends naturally on $\mathds{1}_\mathcal A\in \mathcal A$ up to % isomorphisms: for any isomorphism $\mathds{1}_\mathcal A \cong \mathds{1}_\mathcal A'$ of the preferred objects, we have an equivalence 
%	\[
%	(\mathcal A, \mathds{1}_\mathcal A)^{\otimes Z} \cong (\mathcal A, \mathds{1}_\mathcal A')^{\otimes Z}.
%	\]
%	Therefore, infinite tensor products of based $R$-linear categories are functorial with respect to based $R$-linear functors up to these equivalences.
	
	\begin{example}
		Let $S$ be an $R$-algebra and let $Z$ be a (possibly infinite) set. With respect to the rank one module $S\in \mathcal F^f(S)$ as a base point, we have an equivalence
		\[(\mathcal F^f(S)^{\otimes Z})^f\simeq \mathcal F^f(S^{\otimes Z}),\]
		determined by the functors 
		\[\mathcal F^f(S)^{\otimes F}\to \mathcal F^f(S^{\otimes Z}),\quad M\mapsto M\otimes S^{\otimes Z\setminus F}.\]
	\end{example}

	The \emph{idempotent completion} $\Idem(\mathcal A)$ of an $R$-additive category $\mathcal A$ is the $R$-additive category whose objects are pairs $(A,p)$ with $A\in \mathcal A$ and $p=p^2\in \mathcal A(A,A)$ and whose morphisms $(A,p)\to (B,q)$ are given by morphisms $f\colon A\to B$ satisfying $f=qfp$. We call an $R$-additive category $\mathcal A$ \emph{idempotent complete} if the canonical functor $\mathcal A\to \Idem(\mathcal A),\quad A\mapsto (A,\id)$ is an equivalence. 
	If $\mathcal A$ is a not necessarily additive $R$-linear category, we define $\Idem(\mathcal A)\coloneqq \Idem(\mathcal A^f)$. 
	\begin{example} Let $S$ be an $R$-algebra with unit. Denote by $\mathcal F^p(S)$ the category of finitely generated projective $S$-modules. Then the functor 
			\[\Idem(\mathcal F^f(S))\to \mathcal F^p(S),\quad (M, p) \mapsto pM\]
		is an equivalence. 
	\end{example}

	The following lemma tells us that we can perform additive and idempotent completion at various stages of our constructions or at the end, without changing the result. 

	\begin{lemma}\label{lem:idempotentcomplete}
		The following hold true.
		\begin{enumerate}
			\item Let $\mathcal A,\mathcal B$ be $R$-linear categories. Then the natural functor \label{item:tensoridem0}
			\[(\mathcal A\otimes \mathcal B)^f\to (\mathcal A^f\otimes \mathcal B^f)^f\]
			is an equivalence. 
			\item Let $(\mathcal A_i)_i$ be a filtered diagram of $R$-linear categories and $R$-linear functors. Then the natural functor 
			\[(\varinjlim \mathcal A_i)^f\to (\varinjlim \mathcal A_i^f)^f\]
			is an equivalence. 
			\item Let $G$ be a group and let $\mathcal A$ be a $G$-$R$-linear category. Then the natural functor 
			\[(\mathcal A\rtimes G)^f\to (\mathcal A^f\rtimes G)^f\]
			is an equivalence. 
			
			\item Let $\mathcal A,\mathcal B$ be $R$-linear categories. Then the natural functor \label{item:tensoridem}
			\[\Idem(\mathcal A\otimes \mathcal B)\to \Idem(\Idem(\mathcal A)\otimes \Idem(\mathcal B))\]
			is an equivalence. 
			\item Let $(\mathcal A_i)_i$ be a filtered diagram of $R$-linear categories and $R$-linear functors. Then the natural functor 
				\[\Idem(\varinjlim \mathcal A_i)\to \Idem(\varinjlim \Idem(\mathcal A_i))\]
				is an equivalence. 
			\item Let $G$ be a group and let $\mathcal A$ be a $G$-$R$-linear category. Then the natural functor 
				\[\Idem(\mathcal A\rtimes G)\to \Idem(\Idem(\mathcal A)\rtimes G)\]
				is an equivalence. 
		\end{enumerate}
	\end{lemma}
	\begin{proof}
		It is routine to check that all the appearing functors are fully faithful and essentially surjective. We leave the details to the reader.
%		We only prove \ref{item:tensoridem} since the proofs of the other statements are similar. It is easy to see that the functor $\Idem(\mathcal A\otimes \mathcal B)\to \Idem(\Idem(\mathcal A)\otimes \Idem(\mathcal B))$ is fully faithful. To see that it is also essentially surjective, recall that an object in $\Idem(\Idem(\mathcal A)\otimes \Idem(\mathcal B))$ given by a pair 	
%				\[\left(\left((A_i,p_i)\otimes (B_i,q_i)\right)_{i=1}^n,P\right)\]
%			where $(A_i,p_i)\in \Idem(\mathcal A_i), (B_i,q_i)\in \Idem(\mathcal B_i)$, and where $P$ satisfies 
%				\[\begin{pmatrix}
%				p_1\otimes q_1& &\\
%				&\ddots&\\
%				&&p_n\otimes q_n
%				\end{pmatrix}P\begin{pmatrix}
%				p_1\otimes q_1& &\\
%				&\ddots&\\
%				&&p_n\otimes q_n
%				\end{pmatrix}=P.\]
%			Since the identity of this object is given by $P$, the canonical morphism
%				\[P\colon\left(\left((A_i,p_i)\otimes (B_i,q_i)\right)_1^n,P\right)\to \left(\left((A_i,\id)\otimes (B_i,\id)\right)_1^n,P\right)\]
%			is an isomorphism in $\Idem(\Idem(\mathcal A)\otimes \Idem(\mathcal B))$.
	\end{proof}

	For a finite set $I$ and a set of $R$-linear functors $F_i\colon \mathcal A\to \mathcal B$, we define their direct sum as the functor $F=\bigoplus_{i\in I}F_i\colon \mathcal A\to \mathcal B^f$ by $A\mapsto (F_i(A))_{i\in I}$. If $\mathcal B$ is additive, we may identify $F$ with a functor $\mathcal A\to \mathcal B$. If $I=\{1,\dotsc,n\}$, we may also write $F=F_1\oplus \dotsb \oplus F_n$, well-aware that this notation is imprecise and does not specify the choice of direct sums. 
	The following definition is our key substitute of $C^*$-algebraic $KK$-equivalences in the world of additive categories and algebraic $K$-theory.

	\begin{definition}\label{defn:virtualeq}
		An $R$-linear functor $F\colon \mathcal A\to \mathcal B$ between $R$-additive categories is called a \emph{virtual equivalence}, if there are $R$-linear functors \[{Q_+,Q_-\colon \mathcal B\to \mathcal A}\]
		and natural equivalences 
		\[\eta\colon FQ_+\simeq \id_\mathcal B \oplus FQ_-,\quad \xi\colon Q_+F\simeq \id_\mathcal A\oplus Q_-F.\]
		In this case, we call $(Q_+,Q_-,\xi,\eta)$ a \emph{quasi-inverse} to $F$. 
		An equivariant $R$-linear functor between $G$-$R$-linear categories is called an equivariant virtual equivalence if it has a quasi-inverse $(Q_+,Q_-,\xi,\eta)$ such that $Q_+,Q_-, \xi$ and $\eta$ are equivariant.
		More generally, an $R$-linear functor between $R$-linear categories is called a virtual equivalence if it is a virtual equivalence after additive completion. 
	\end{definition}
	
	Note that any virtual equivalence induces an isomorphism on $K$-theory. 
	
	\begin{remark}
		Let $F\colon \mathcal A\to \mathcal B$ be an equivariant virtual equivalence of $G$-$R$-linear categories. Then $F\rtimes G\colon \mathcal A\rtimes G\to \mathcal B\rtimes G$ is a virtual equivalence.
	\end{remark}
	
	For the next example, recall that for two unital rings $S, S'$, a right $S$-module $M$, and a (not necessarily unital) ring homomorphism $f\colon S\to S'$, we denote by 
	$\Ind_fM$ the right $S'$-module $M\otimes_S f(1)S'$ where $f(1)S'$ is considered as an $S$-$S'$-bimodule via $f$. Note that $\Ind_f$ preserves finitely generated projective modules and that $\Ind_f$ may not preserve free modules if $f$ is not unital. 
	
	\begin{example}\label{eg:finite group ring}
		Let $G$ be a finite group and suppose that $\frac 1 {|G|}\in R$. We denote by 
			\[\varepsilon\colon R[G]\to R,\quad \varepsilon\left(\sum_{g\in G} f_g g\right)=\sum_{g\in G}f_g\]
		the augmentation map, by $I_R[G]\coloneqq \ker \varepsilon$ the augmentation ideal, and by $p\coloneqq \frac 1 {|G|}\sum_{g\in G}(1-g)$ its unit. Note that we have an algebra isomorphism $R\oplus I_R[G] \cong R[G]$, and thus an equivalence $F^p(R)\oplus \mathcal F^p(I_R[G])\cong \mathcal F^p(R[G])$. Below, we give a different functor $F^p(R)\oplus \mathcal F^p(I_R[G])\to \mathcal F^p(R[G])$ which is not necessarily an equivalence but a virtual equivalence. Denote by $j\colon I_R[G]\to R[G]$ the inclusion, by $\pi\colon R[G]\to I_R[G],\quad x\mapsto px$ the projection and by $\iota\colon R\to R[G]$ the unital inclusion. The functor 
		%		\[F\coloneqq \Res_\iota\oplus \Res_j\colon \mathcal F^f(R[G])\to \mathcal F^f(R)\oplus \mathcal F^f(I_R[G])\]
		\[F\coloneqq \Ind_\iota\oplus \Ind_j\colon  \mathcal F^p(R)\oplus \mathcal F^p(I_R[G])\to \mathcal F^p(R[G])\]
		sends the rank-one free $R$-module $R$ to the rank-one free $R[G]$-module $R[G]$, and the rank-one free $I_R[G]$-module $I_R[G]$ to the projective $R[G]$-module $I_R[G]\otimes_{I_R[G]}R[G]\cong I_R[G]$. The functor $F$ is a virtual equivalence with quasi-inverse 
		\[(\Ind_\varepsilon\oplus \Ind_\pi,0\oplus \Ind_{\pi\iota\varepsilon}).\]
			%Then the $R$-linear map $\iota\oplus j\colon R\oplus I_R[G]\to R[G]$ is invertible with inverse $\varepsilon\oplus \pi-0\oplus \pi\iota\varepsilon$. As a consequence, 
	\end{example}

	\section{Izumi's theorem}
	%\comment{At some point we should put a definition or reference to a definition for algebraic $K$-theory and specify if we deal with connective or non-connective $K$-theory.}
	This section is devoted to an algebraic $K$-theory analogue of Izumi's \cite[Theorem 2.1]{Izumi} which says that if $A$ and $B$ are $KK$-equivalent separable $C^*$-algebras and $G$ is a finite group, then the Bernoulli shifts $A^{\otimes G}$ and $B^{\otimes G}$ are $KK^G$-equivalent. Our substitute for $KK$-equivalences between $C^*$-algebras is the notion of virtual equivalences between additive categories introduced in Definition \ref{defn:virtualeq}. 
	For a group $G$, a $G$-set $Z$, and a (based) $R$-linear category $\mathcal A$, we always equip the tensor product $\mathcal A^{\otimes Z}$ with the $G$-action given by shifting the tensor factors. 
	\begin{theorem}[Izumi's Theorem]\label{lem:virtualEqFiniteGroup}
		Let $G$ be a finite group, let $Z$ be a finite $G$-set, and let $F\colon \mathcal A\to \mathcal B$ be an $R$-linear virtual equivalence of $R$-additive categories. Then 
			\[F^{\otimes Z}\colon \mathcal A^{\otimes Z}\to \mathcal B^{\otimes Z}\]
		is an equivariant virtual equivalence. In particular,
			\[F^{\otimes Z}\rtimes G\colon \mathcal A^{\otimes Z}\rtimes G\to \mathcal B^{\otimes Z}\rtimes G\]
		is a virtual equivalence.
	\end{theorem}

	We keep the notation of Theorem \ref{lem:virtualEqFiniteGroup} and divide its proof into several lemmas. Let $(Q_+,Q_-,\xi,\eta)$ be a quasi-inverse for $F$. 
	We will construct a pair of functors $P_\pm\colon \mathcal B^{\otimes Z} \to (\mathcal A^{\otimes Z})^f$ and a $G$-equivariant natural equivalence of functors 
	\[P_+F^{\otimes Z}\simeq \iota\oplus P_-F^{\otimes Z}\colon \mathcal A^{\otimes Z}\to (\mathcal A^{\otimes Z})^f,\]
	where $\iota\colon \mathcal A^{\otimes Z}\to (\mathcal A^{\otimes Z})^f$ is the inclusion. The proof of the equivalence $F^{\otimes Z}P_+\simeq \iota\oplus F^{\otimes Z}P_-$ is analogous and left to the reader. 
	\begin{definition}
	For any subset $T\subseteq Z$ and for any subgroup $H\subseteq G_T$ of the setwise stabilizer group of $T$, we define $G$-equivariant functors 
	\[Q_{H,T}\coloneqq \bigoplus_{gH\in G/H}Q_+^{\otimes Z\setminus gT}\otimes Q_-^{\otimes gT}\colon  \mathcal B^{\otimes Z}\to (\mathcal A^{\otimes Z})^f,\]
	\[R_{H,T}\coloneqq \bigoplus_{gH\in G/H}\id^{\otimes Z\setminus gT}\otimes (Q_-F)^{\otimes gT}\colon  \mathcal A^{\otimes Z}\to (\mathcal A^{\otimes Z})^f,\]
	where the $G$-action on the index set $G/H$ is given by left translation. 
	\end{definition}

These formulas may initially appear overwhelming, so we explain their origins. Considering the informal expression \( Q_+ - Q_- \) as the quasi-inverse of \( F \), and leveraging the binomial formula, an initial guess for the quasi-inverse of \( F^{\otimes Z} \) could be constructed by expanding \( (Q_+ - Q_-)^{\otimes Z} \). Specifically, using a combinatorial perspective, we might propose:

\[
(Q_+ - Q_-)^{\otimes Z} = \left( \bigoplus_{\substack{S \subset Z \\ \text{even}}} Q_+^{\otimes (Z \setminus S)} \otimes Q_-^{\otimes S} \right) 
- \left( \bigoplus_{\substack{S \subset Z \\ \text{odd}}} Q_+^{\otimes (Z \setminus S)} \otimes Q_-^{\otimes S} \right),
\]

where the sums are taken over subsets \( S \subset Z \) with even and odd cardinalities, respectively.

However, this expression turns out not to be a valid (\( G \)-equivariant) quasi-inverse of \( F^{\otimes Z} \). The simplest counterexample arises when \( G \) is the cyclic group \( C_2 \) of order 2, \( Z = G \), \( F = \mathrm{id} \), \( Q_+ = \mathrm{id} \oplus \mathrm{id} \), and \( Q_- = \mathrm{id} \) on \( \mathcal{A} = (\underline{\mathbb{C}})^f \), for example. In this case, the expansion does not satisfy the required properties of a quasi-inverse.

The core issue lies in two subtleties: the binomial formula in this setting must account for \( G \)-equivariant indices, which introduces additional complications when we consider the product of such formulas; and the power functor \( (-)^{\otimes Z} \) is not linear, meaning it does not distribute over sums or differences in the straightforward way that one might expect.

Thus, while the proposed expansion might appear intuitive at first glance, we need a more careful treatment to correctly handle the \( G \)-equivariance and the nonlinearity of the power functor. The previously defined \( Q_{H, T} \) accounts for a more general summand of the binomial formula, whose suitable combinations provide the correct quasi-inverse of \( F^{\otimes Z} \).

	\begin{lemma}\label{lem:distributive}
		For any subset $T\subseteq Z$ and for any subgroup $H\subseteq G_T$, we have a $G$-equivariant natural equivalence of functors 
			\[Q_{H,T}F^{\otimes Z}\simeq \bigoplus_{\underset{ S\supset gT}{(gH,S)\in G/H\times \mathcal P(Z),}}\id^{\otimes Z\setminus S}\otimes (Q_-F)^{\otimes S}\colon \mathcal A^{\otimes Z}\to (\mathcal A^{\otimes Z})^f,\]
		where $\mathcal P(Z)$ denotes the power set of $Z$. Here the $G$-action on $G/H\times \mathcal P(Z)$ is induced by left translation on $G/H$ and the given $G$-action on $Z$. 
	\end{lemma}
\begin{proof}
	We have an equivariant natural equivalence of functors 
	\[\bigoplus_{gH\in G/H}\xi^{\otimes Z\setminus gT}\otimes \id^{\otimes gT}\colon Q_{H,T}F^{\otimes Z}\simeq \bigoplus_{gH\in G/H}(\id \oplus Q_-F)^{\otimes Z\setminus gT}\otimes (Q_-F)^{\otimes gT},\]
	where $\xi$ is the given equivalence $Q_+F\simeq \id_\mathcal A\oplus Q_-F$ from the quasi-inverse of $F$. 
	Using that the direct sum is both a product and a coproduct, we construct for every $gH\in G/H$ a natural equivalence
\begin{equation}\label{eq:Qid=Q}
	 \eta_{gH}\colon \bigoplus_{S\supset gT}\id^{\otimes Z\setminus S}\otimes (Q_-F)^{\otimes S}\simeq (\id \oplus Q_-F )^{\otimes Z\setminus gT}\otimes (Q_-F)^{\otimes gT}.
\end{equation}
	given by 
	\[\eta_{gH}\coloneqq \bigoplus_{S\supset gT} (\id\oplus0)^{\otimes Z\setminus S}\otimes (0\oplus \id)^{\otimes S\setminus gT}\otimes \id^{\otimes gT}.\]
	For $g\in G$ and $A=\otimes_{z\in Z}A_z\in \mathcal A^{\otimes Z}$, we get 
	\begin{align*}
		&\tilde \alpha_g\left(\bigoplus_{hH\in G/H}\eta_{hH}\left(A\right)\right)\\
		&=\tilde \alpha_g \left(\bigoplus_{hH\in G/H}\bigoplus_{S\supset hT}(\id\oplus 0)^{\otimes Z\setminus S}\otimes (0\oplus \id)^{S\setminus hT}\otimes \id^{hT}\right)\\ 
		&=\bigoplus_{hH\in G/H}\bigoplus_{S\supset hT} \alpha_g \left((\id\oplus 0)^{\otimes Z\setminus g^{-1}S}\otimes (0\oplus \id)^{g^{-1}S\setminus g^{-1}hT}\otimes \id^{g^{-1}hT}\right)\\
		&=\bigoplus_{hH\in G/H}{\eta_{hH}}\left(\bigotimes_{z\in Z}A_{g^{-1}z}\right)\\
		&=\bigoplus_{hH\in G/H}\eta_{hH}(\alpha_g(A)),
	\end{align*}
	where $\alpha$ denotes the shift action on $\mathcal A^{\otimes Z}$. 	
	In particular, the induced natural equivalence
	\[\bigoplus_{gH\in G/H}\eta_{gH}\colon Q_{H,T}F^{\otimes Z}\simeq \bigoplus_{\underset{ S\supset gT}{(g,S)\in G/H\times \mathcal P(Z),}}\id^{\otimes Z\setminus S}\otimes (Q_-F)^{\otimes S}\]
	is equivariant.
\end{proof}

\begin{remark}
	Our main reason for introducing non-trivial $G$-actions on the index sets in $(\mathcal A^{\otimes Z})^f$ was to ensure that Lemma \ref{lem:distributive} holds equivariantly.
\end{remark}

\begin{lemma}\label{lem:remainder}
	For any subset $T\subseteq Z$ and for any subgroup $H\subseteq G_T$, we have an equivariant natural equivalence 
		\[Q_{H,T}F^{\otimes Z}\simeq R_{H,T}\oplus R_0\]
	where $R_0$ is a direct sum of finitely many functors of the form $R_{H',T'}$ with $|T'|>|T|$. 
\end{lemma}
\begin{proof}
	By Lemma \ref{lem:distributive}, we have 
	\begin{align*}
		Q_{H,T}F^{\otimes Z}&\simeq \bigoplus_{\underset{ S\supset gT}{(gH,S)\in G/H\times \mathcal P(Z)}}\id^{\otimes Z\setminus S}\otimes (Q_-F)^{\otimes S}\\
		&= R_{H,T}\oplus \underbrace{\bigoplus_{\underset{ S\supsetneqq gT}{(gH,S)\in G/H\times \mathcal P(Z)}}\id^{\otimes Z\setminus S}\otimes (Q_-F)^{\otimes S}}_{\eqqcolon R_0}.
	\end{align*}
	By decomposing the $G$-set $\{(gH,S)\mid gH\in G/H, S\supsetneqq gT\}$ into orbits we can find $S_1,\dotsc,S_n\supsetneqq T$ such that 
		\[\{(gH,S)\mid gH\in G/H, S\supsetneqq gT\}= \bigsqcup_{i=1}^n G\cdot\{(H,S_i)\}.\]
	We write $H_i\coloneqq H\cap G_{S_i}$ and conclude that 
		\[R_0\simeq R_{H_1,S_1}\oplus \dotsb \oplus R_{H_n,S_n}.\]
	
\end{proof}

\begin{proof}[Proof of Theorem \ref{lem:virtualEqFiniteGroup}]
	We construct equivariant functors $P_\pm \colon \mathcal B^{\otimes Z}\to (\mathcal A^{\otimes Z})^f$ and an equivariant natural equivalence $P_+F^{\otimes Z}\simeq\iota \oplus P_-F^{\otimes Z}$ where $\iota\colon \mathcal A^{\otimes Z}\to (\mathcal A^{\otimes Z})^f$ is the natural inclusion. An analogous argument will produce a natural equivalence $F^{\otimes Z}P_+\simeq\iota \oplus F^{\otimes Z}P_-$. The theorem then follows by applying $-\rtimes G$ to these equivalences and composing everything with the equivalence between $((\mathcal A^{\otimes Z})^f\rtimes G)^f\simeq (\mathcal A^{\otimes Z}\rtimes G)^f$ and $((\mathcal B^{\otimes Z})^f\rtimes G)^f\simeq (\mathcal B^{\otimes Z}\rtimes G)^f$ from Lemma \ref{lem:idempotentcomplete}.
	
	By Lemma \ref{lem:remainder}, we have a $G$-equivariant natural equivalence 
	\begin{equation}\label{QRR0}
		Q_{G,\emptyset}F^{\otimes Z}\simeq R_{G,\emptyset}\oplus R^{(1)},
	\end{equation}
	where $R^{(1)}$ is of the form $R^{(1)}=\bigoplus_{i=1}^{n_1}R_{H^{(1)}_i,S^{(1)}_i}$ with $|S^{(1)}_i|\geq 1$. 
	For each $i=1,\dotsc,n_1$, Lemma \ref{lem:remainder} gives us a $G$-equivariant natural equivalence 
		\begin{equation}\label{eq:QRR1}
			Q_{H^{(1)}_i,S^{(1)}_i}F^{\otimes Z}\simeq R_{H^{(1)}_i,S^{(1)}_i}\oplus R^{(2)}_i
		\end{equation}
	where each $R^{(2)}_i$ is a direct sum of functors of the form $R_{H,S}$ with $|S|\geq 2$. 
	We write $R^{(2)}=\bigoplus_{i=1}^{n_1}R^{(2)}_i= \bigoplus_{i=1}^{n_2} R_{H^{(2)}_i,S^{(2)}_i}$. By summing up \eqref{eq:QRR1} over $i=1,\dotsc,n_1$, we get 
		\begin{equation}\label{eq:QRR2}
			\bigoplus_{i=1}^{n_1} Q_{H^{(1)}_i,S^{(1)}_i}F^{\otimes Z}\simeq R^{(1)}\oplus R^{(2)}.
		\end{equation}
	We now add $R^{(2)}$ to \eqref{QRR0} and apply \eqref{eq:QRR2} to get an equivariant natural equivalence 
	\begin{equation}\label{eq:QRR3}
		 Q_{G,\emptyset}F^{\otimes Z}\oplus R^{(2)}\simeq R_{G,\emptyset}\oplus \bigoplus_{i=1}^{n_1} Q_{H^{(1)}_i,S^{(1)}_i}F^{\otimes Z}.
	\end{equation}
	In the same way as in \eqref{eq:QRR2}, we get 
		\begin{equation}\label{eq:QRR4}
		\bigoplus_{i=1}^{n_2} Q_{H^{(2)}_i,S^{(2)}_i}F^{\otimes Z}\simeq R^{(2)}\oplus R^{(3)},
	\end{equation}
	where $R^{(3)}$ is a direct sum of functors of the form $R_{H,S}$ with $|S|\geq 3$.
	We again add $R^{(3)}$ to \eqref{eq:QRR3} and apply \eqref{eq:QRR4} to get 
	\begin{equation}\label{eq:QRR5}
		Q_{G,\emptyset}F^{\otimes Z}\oplus \bigoplus_{i=1}^{n_2}Q_{H^{(2)}_i,S^{(2)}_i}F^{\otimes Z}\simeq R_{G,\emptyset}\oplus \bigoplus_{i=1}^{n_1} Q_{H^{(1)}_i,S^{(1)}_i}F^{\otimes Z}\oplus R^{(3)}.
	\end{equation}
	We continue this process inductively and construct $R^{(k)},n_k, H_i^{(k)}$ and $S_i^{(k)}$ as above for $k=1,2,\dotsc, |Z|$. 
	Now note that $R_{H,Z}=Q_{H,Z}F^{\otimes Z}$ for any subgroup $H\subseteq G$ and that $R_{G,\emptyset}=\iota$. If we conveniently write $Q_{G,\emptyset}=Q_{H_1^{(0)},S_1^{(0)}}$ and apply the above produre until $k=|Z|$, we therefore get an equivariant natural equivalence
	\[\bigoplus_{\underset{\text{even}}{0\leq k \leq |Z|}}\bigoplus_{i=1}^{n_k}Q_{H_i,S_i}F^{\otimes Z}\simeq\iota \oplus \bigoplus_{\underset{\text{odd}}{0\leq k \leq |Z|}}\bigoplus_{i=1}^{n_k}Q_{H_i,S_i}F^{\otimes Z}.\]
	The desired functors $P_\pm$ may thus be chosen as  
	\[P_+\coloneqq \bigoplus_{\underset{\text{even}}{0\leq k \leq |Z|}}\bigoplus_{i=1}^{n_k}Q_{H_i,S_i},\quad P_-\coloneqq \bigoplus_{\underset{\text{odd}}{0\leq k \leq |Z|}}\bigoplus_{i=1}^{n_k}Q_{H_i,S_i}.\]
\end{proof}

\section{$K$-theory of wreath products}	\label{sec-wreath}
	
	Recall that a ring $S$ is called \emph{Noetherian} if every submodule of a finitely generated $S$-module is finitely generated, \emph{regular coherent} if every finitely presented $S$-module has a finite projective resolution, and \emph{regular} if it is Noetherian and regular coherent.
	Analogous definitions were introduced in \cite[Definition 6.2]{bartels2022vanishing} for any additive category $\mathcal A$.
	We recall these for the reader's convenience. A right $\Z\mathcal A$-module (or simply $\Z\mathcal A$-module) is a contravariant $\Z$-linear functor $\mathcal A\to \Ab$. We denote the category of $\mathbb Z\mathcal A$-modules with $\Z$-linear transformations as morphisms by $\Mod \Z\mathcal A$. 
	A $\Z \mathcal A$-module $M$ is called \emph{finitely generated free} if it is isomorphic to a direct sum $\bigoplus_{i\in I}\Hom_\mathcal A(-, A_i)$, and \emph{finitely generated} if it is isomorphic to the quotient of a finitely generated free $\Z \mathcal A$-module. A $\Z\mathcal A$-module $A$ is called \emph{finitely presented} if there is an exact sequence $F_1\to F_0\to M\to 0$ where $F_0$ and $F_1$ are finitely generated free $\Z\mathcal A$-modules. A $\Z\mathcal A$-module is finitely generated projective if it is a retract of a finitely generated free $\Z\mathcal A$-module\footnote{We encourage the reader to verify that this characterization is indeed equivalent to the definition projectivity in the sense of abelian categories.}. A \emph{finite projective resolution} of a $\Z\mathcal A$-module $M$ is an exact sequence $0\to P_n\to \dotsc\to P_0\to M\to 0$ where each $P_i$ is a finitely generated projective $\Z\mathcal A$-module. 
	We call $\mathcal A$ \emph{Noetherian} if every submodule of a finitely generated $\Z\mathcal A$-module is finitely generated, \emph{regular coherent} if every finitely presented $\Z\mathcal A$-module has a finite projective resolution, and \emph{regular} if it is Noetherian and regular coherent.	
	A ring is Noetherian (resp. regular, resp. regular coherent) if and only if the category of finitely generated free (or equivalently projective) modules is \cite[Corollary 6.5]{bartels2022vanishing}. 
%	For the following lemma is well-known but I don't know the correct reference. It is recorded in \cite[Lemma 5.7]{Grunewald} but this is certainly not the first reference. Note that the corresponding result does not hold for regular coherent rings, see \cite[Lemma 4.3]{BartelsLueckKH}.
%	
%	\begin{lemma}
%		Let $R$ be a regular ring and let $H$ be a finite group such that $\frac 1 {|H|}\in R$. Then $R[H]$ is regular. 
%	\end{lemma}
%	\begin{proof}
%		Let $M$ be a finitely generated $R[H]$-module. Then $R$ is finitely generated as an $R$-module and has a finite projective resolution $P_\bullet$ over $R$
%		Since $R[H]$ is $R$-flat, we get a finite projective resulotion $P_\bullet\otimes R[H]$ of $M\otimes R[H]$ over $R[H]$. Now the maps 
%		\[M\to M\otimes R[H],\quad m\mapsto \frac 1 {|H|}\sum_{h\in H}mh^{-1}\otimes h\]
%		\[M\otimes R[H]\to M,\quad m\otimes h\mapsto mh \]
%		exhibit $M$ as a split summand of $M\otimes R[H]$ over $R[H]$. Pulling back this splitting yields a finite projective resolution of $M$ over $R[H]$. 		
%	\end{proof}
%	
%	\begin{remark}
%		The same proof shows that if $\frac 1 {|H|}\in R$ and if $\mathcal A$ is a regular $R$-additive category (in the sense of \cite{bartels2022vanishing}), then $\mathcal A\rtimes H$ is again regular. 
%	\end{remark}

	The following lemma is well-known for rings. We adapt the proof of \cite[Lemma 5.7]{Grunewald} to the setting of additive categories. Recall that a functor between abelian categories is called \emph{exact} if it preserves exact sequences. We call a functor $\iota\colon \mathcal A \to \mathcal B $ of additive categories  \emph{flat} if the induction functor $\iota_\ast\colon \Mod \Z \mathcal A\to \Mod \Z \mathcal B$ is exact.  
	\begin{lemma}\label{lem:regular}
		Let $G$ be a finite group such that $\frac 1 {|G|}\in R$ and let $\mathcal A$ be a regular $G$-$R$-additive category. Then $\mathcal A\rtimes G$ is regular as well. 
	\end{lemma}
	\begin{proof}
		We denote by $\iota\colon \mathcal A\to \mathcal A\rtimes G$ the canonical inclusion and by 
		\[\iota^*\colon \Mod \Z(\mathcal A\rtimes G)\to \Mod \Z\mathcal A,\quad M\mapsto M\circ \iota\]
		\[\iota_*\colon \Mod \Z \mathcal A\to \Mod \Z(\mathcal A\rtimes G),\quad M\mapsto M(?)\otimes_{\Z\mathcal A}\Hom_{\mathcal A\rtimes G}(-,\iota(?))\]
		the canonical restriction and induction functors. 
		Here, $M(?)\otimes_{\Z\mathcal A}\Hom_{\mathcal A\rtimes G}(A,\iota(?))$ is the quotient of 
		\[\bigoplus_{B\in \mathcal A}M(B)\otimes_{\Z}\Hom_{\mathcal A\rtimes G}(A,\iota(B))\]
		by the submodule generated by elements of the form $f^*(x)\otimes y-x\otimes f_*( y)$ for $f\in \Hom_{\mathcal A}(B_1,B_2), x\in M(B_2)$, and $y\in \Hom_{\mathcal A\rtimes G}(A,\iota(B_1))$. 
		
		The restriction functor $\iota^*$ is exact. The induction functor $\iota_*$ is exact as well, i.e. $\iota$ is flat (see below).
		
		\begin{claim}\label{claim:splitsummand}
		Every $\mathbb Z(\mathcal A\rtimes G)$-module $M$ is a summand of $\iota_*\iota^*M$ in $\Mod \Z (\mathcal A\rtimes G)$. 
		\end{claim}
		The desired inclusion $j\colon M\to \iota_*\iota^* M$ is given by 
		\[M(A)\ni x\mapsto \frac 1 {|G|} \sum_{g\in G} g^*(x)\otimes g^{-1}\in \iota_*\iota^*M(A)\]
		where 
		\[g^*(x)\otimes g^{-1}\in M(\alpha_{g^{-1}}(A))\otimes_\Z \Hom_{\mathcal A\rtimes G}(A,\alpha_{g^{-1}}(A)).\]
		To see that $j$ is a module map, let $f\in \Hom_\mathcal A(A,B)$ and $x\in M(B)$. Then we have 
		\begin{align*}
			f^*(j(x))&= \frac 1 {|G|}\sum_{g\in G} g^*(x)\otimes g^{-1}\circ f\\
			&= \frac 1 {|G|}\sum_{g\in G} g^*(x)\otimes \alpha_{g^{-1}}(f)\circ g^{-1}\\
			&= \frac 1 {|G|}\sum_{g\in G} \alpha_{g^{-1}}(f)^*g^*(x)\otimes g^{-1}\\
			&= \frac 1 {|G|}\sum_{g\in G} g^*(f^*(x))\otimes g^{-1}	\\
			&= j(f^*(x)).
		\end{align*}
		The desired retraction is given by the co-unit 
		$\pi\colon \iota_*\iota^*M(A)\to M(A)$
		which is given by 
		\[ M(B)\otimes_\Z \Hom_{\mathcal A\rtimes G}(A,B)\ni x\otimes f\mapsto f^*(x)\in M(A). \]
		It is easy to see that $\pi\circ j=\id$. This verifies Claim \ref{claim:splitsummand}.
	
		\begin{claim}\label{claim:iotafg}
			$\iota^*$ preserves finitely generated free modules. 
		\end{claim}
		Note that for any object $A\in \mathcal A$, we have a canonical isomorphism of $\Z\mathcal A$-modules
		\begin{equation}\label{eq:finitelygenerated}
		\iota^*\Hom_{\mathcal A\rtimes G}(-,A)\cong \bigoplus_{g\in G} \Hom_\mathcal A(\alpha_g(-),A)\cong \bigoplus_{g\in G}\Hom_\mathcal A(-,\alpha_{g^{-1}}(A)).
		\end{equation}
		This proves Claim \ref{claim:iotafg}.
		
		We now explain why $\iota_*$ is exact. Since $\iota$ is essentially surjective, it is enough to see its restriction $\iota^{\ast}\iota_\ast $ is exact. For any $\Z\mathcal A$-module $M$, we have a natural isomorphism
		\[
	(\iota^{\ast}\iota_\ast M)(A) \cong \oplus_{g\in G}M(\alpha_g(A)) 
		\]
for any object $A$ in $\mathcal A$. This is obtained by sending
\[
x\otimes f_gg  \in M(B)\otimes_{\Z}\Hom_{\mathcal A\rtimes G}(\iota(A),\iota(B))
\]
for $f_g\in \Hom_{\mathcal{A}}(\alpha_g(A), B)$ to 
\[
f_g^*x \in M(\alpha_g(A))
\]
We leave it to the reader to verify that this is well-defined and natural, establishing the natural isomorphism $\iota^{\ast}\iota_\ast M \cong \oplus_{g\in G}M(\alpha_g-)$. With this in hand, it follows that $\iota_\ast$ is exact.

		To see that $\mathcal A\rtimes G$ is Noetherian, let $N\subseteq M$ be a submodule of a finitely generated $\Z(\mathcal A\rtimes G)$-module. 
		Then $\iota^*N\subseteq \iota^*M$ is finitely generated by Claim \ref{claim:iotafg} and by the exactness of $\iota^*$. Since $\mathcal A$ is Noetherian by assumption, we find an epimorphism $F\twoheadrightarrow \iota^* N$ where $F$ is a finitely generated free $\Z\mathcal A$-module. Since $\iota_*$ is right exact (being a left adjoint) and preserves finitely generated free modules (being a left Kan extension, see also the comments at the end of \cite[Section 5.A]{bartels2022vanishing}), we get an epimorphism $\iota_*F\twoheadrightarrow \iota_*\iota^* N\twoheadrightarrow N$, where $\iota_*F$ is a finitely generated $\Z(\mathcal A\rtimes G)$-module. This implies that $N$ is finitely generated and therefore that $\mathcal A\rtimes G$ is Noetherian.

		To see that $M$ is also regular coherent, let $M\in \Mod \Z(\mathcal A\rtimes G)$ be finitely generated. 
		Since $\iota^*M\in \Mod \Z\mathcal A$ is finitely generated by Claim \ref{claim:iotafg}, we can find a finite projective resolution $P_\bullet\to \iota^*M$. 
		Then $\iota_*P_\bullet\to \iota_*\iota^*M$ is a finite projective resolution of $\Z(\mathcal A\rtimes G)$-modules. 
		Since $M$ is a summand of $\iota_*\iota^*M$ by Claim \ref{claim:splitsummand}, it has a finite projective resolution as well. Indeed, such a projective resolution can be constructed by iteratively taking the pullback of the resolution $\iota_*P_\bullet\to \iota_*\iota^*M$, starting with the split inclusion $M\to   \iota_*\iota^*M$. The exactness of the resulting sequence can be verified directly, pointwise.
				
	\end{proof}

	\begin{lemma}\label{lem:virtually cyclic}
		Let $H$ be a finite group such that $\frac 1 {|H|}\in R$ and let $\varphi\in \Aut(H)$ be an automorphism. Let $Z$ be a $G$-set, where $G=H\rtimes_\varphi \Z$. Let $F\colon \mathcal A\to \mathcal B$ be a based $R$-linear functor of based $R$-additive categories which is furthermore a virtual equivalence. Assume further that $\mathcal A^{\otimes n}$ and $\mathcal B^{\otimes n}$ are regular for every $n\geq 1$ and that for every inclusion of finite $H$-invariant subsets $Y\subseteq Y'\subseteq Z$, the functors 
			\[\mathcal A^{\otimes Y}\rtimes H\to \mathcal A^{\otimes Y'}\rtimes H, \quad  \mathcal B^{\otimes Y}\rtimes H\to  \mathcal B^{\otimes Y'}\rtimes H\]
		are flat inclusions. 
		Then the induced map 
		\[K_*(F^{\otimes Z}\rtimes G)\colon K_*(\mathcal A^{\otimes Z}\rtimes G)\to  K_*(\mathcal B^{\otimes Z}\rtimes G)\]
		is an isomorphism.
%		, where $K_*$ denotes connective $K$-theory. 
%		If moreover $A^{\otimes n}[\Z^m]$ and $\mathcal B^{\otimes n}[\Z^m]$ are regular for every $n,m\geq 1$, then the same statement holds for non-connective $K$-theory. 
	\end{lemma}
	\begin{proof}
		We prove the lemma by reducing it to Theorem \ref{lem:virtualEqFiniteGroup}. 
				Denote by $x\in G=H\rtimes_\varphi\Z$ the generator of $\Z$ so that $xhx^{-1}=\varphi(h)$ for all $h\in H$. We denote the shift actions on $\mathcal A^{\otimes Z}$ and $\mathcal B^{\otimes Z}$ by $\alpha$ and $\beta$ respectively. 
		We define an automorphism $\Phi\in \Aut(\mathcal A^{\otimes Z}\rtimes H)$ by $A\mapsto \alpha_x(A)$ on objects and by
			\[\sum_{h\in H}f_h h\mapsto\sum_{g\in G} \alpha_x(f_h)\varphi(h)\]
		on morphisms.
		Then there is a natural isomorphism of $R$-additive categories 
			\begin{equation}\label{eq:semidirectproductdecom}
			 \mathcal A^{\otimes Z}\rtimes(H\rtimes_\varphi \Z)\simeq (\mathcal A^{\otimes Z}\rtimes H)\rtimes_\Phi \Z, 
			\end{equation}
		given by the identity on objects and by
			\[\sum_{(h,n)\in H\rtimes_\varphi \Z}f_{h,n}(h,n)\mapsto \sum_{n\in \Z}\left(\sum_{h\in H}f_{h,n}h\right)  n\]
		on morphisms. %\todo{We should check this.}
		We have an analogous automorphism $\Psi\in \Aut(\mathcal B^{\otimes Z}\rtimes H)$ and an isomorphism
			\begin{equation}\label{eq:semidirectproductdecom2}
			\mathcal B^{\otimes Z}\rtimes (H\rtimes_\phi \Z)\simeq (\mathcal B^{\otimes Z}\rtimes H)\rtimes_\Psi \Z.
			\end{equation}
		Note that $\mathcal A^{\otimes Z}\rtimes H$ is equal to the increasing union $\bigcup_Y A^{\otimes Y}\rtimes H$ where $Y$ ranges over all finite $H$-invariant subsets of $Z$. 
		Each $\mathcal A^{\otimes Y}\rtimes H$ is regular by Lemma \ref{lem:regular}. 
		Since the inclusions $\mathcal A^{\otimes Y}\rtimes H\to \mathcal A^{\otimes Y'}\rtimes H$ are flat, \cite[Lemma 11.2]{bartels2022vanishing} implies that $\mathcal A^{\otimes Z}\rtimes H$ is regular coherent. 
		Combined with \cite[Theorem 10.1]{bartels2022vanishing}, this argument even shows that $(\mathcal A^{\otimes Z}\rtimes H)[\Z^m]$ is regular coherent for any $m\geq 0$, and similarly $(\mathcal B^{\otimes Z}\rtimes H)[\Z^m]$. 
		It then follows from \cite[Theorem 7.8]{bartels2022vanishing} together with \eqref{eq:semidirectproductdecom} and \eqref{eq:semidirectproductdecom2} that we only need to prove that the induced map
			\[K_*(\mathcal A^{\otimes Z}\rtimes H)\to K_*(\mathcal B^{\otimes Z}\rtimes H)\]
		is an isomorphism.
		%\comment{\cite[Theorem 10.1]{bartels2022vanishing} says that if $\mathcal A$ is regular, then $\mathcal A[\Z]$ is regular as well. }
		By compatibility of $K$-theory with filtered colimits (see \cite[Corollary 7.2]{Nonconnective}), we may even replace $Z$ by a finite $H$-set $Y$. But now the statement follows from Theorem \ref{lem:virtualEqFiniteGroup}. 
		%The statement about non-connective $K$-theory follows by applying \cite[Theorem 7.8]{bartels2022vanishing} instead of \cite[Theorem 7.5]{bartels2022vanishing}.
	\end{proof}

	The following theorem is an immediate consequence of \cite[Corollary 1.2, Remark 1.6]{JamesQuinnReich}. 
%	We refer to \cite{bartels2007coefficients} for the formulation of the $K$-theoretic Farrell--Jones conjecture with coefficients and to \cite{FJsurvey} for a survey.
%	The Farrell--Jones conjecture with coefficients holds for many classes of groups, such as hyperbolic groups \cite{Bartelshyperbolic}, virtually solvable groups \cite{FJvirtuallysolvable}, mapping class groups \cite{FJmappingclass}, and many more.
%	 The precise formulation of the conjecture is not important for us, we only need the following consequence:
	\begin{theorem}[\cite{JamesQuinnReich}]\label{thmFJimpliesGoingDown}
		Let $G$ be a group satisfying the $K$-theoretic Farrell--Jones conjecture with coefficients. Let $F\colon \mathcal A\to \mathcal B$ be a $G$-equivariant additive functor between additive categories with $G$-action. Assume that for every finite-by-cyclic\footnote{Recall that a group $V$ is called finite-by-cyclic if there is a short exact sequence $1\to H\to V\to C\to 1$ with $H$ finite and $C$ cyclic.} group $V\subseteq G$, the induced map 
		\[K_*(F\rtimes V)\colon K_*(\mathcal A\rtimes V)\to K_*(\mathcal B\rtimes V)\]
		is an isomorphism. Then the induced map 
		\[K_*(\mathcal A\rtimes G)\to K_*(\mathcal B\rtimes G)\]
		is an isomorphism as well. 
	\end{theorem}

\begin{corollary}\label{cor:grouprings}
	Let $G$ be a group satisfying the $K$-theoretic Farrell--Jones conjecture with coefficients, let $H$ be a finite group, and let $Z$ be a $G$-set. Assume that $R$ is regular and that the orders of $H$ and of all finite subgroups of $G$ are invertible in $R$. Denote by 	
	\[F\coloneqq \Ind_\iota\oplus \Ind_j\colon  \mathcal F^p(R)\oplus \mathcal F^p(I_R[H])\to \mathcal F^p(R[H])\]
	the virtual equivalence from Example \ref{eg:finite group ring}, viewed as a based virtual equivalence with respect to the rank one modules $R\in \mathcal F^p(R), R[H]\in \mathcal F^p(R[H])$. Then the induced map 
		\[K_*(F^{\otimes Z}\rtimes G)\colon K_*((\mathcal F^p(R)\oplus \mathcal F^p(I_R[H]))^{\otimes Z}\rtimes G)\to K_*(\mathcal F^p(R[H])^{\otimes Z}\rtimes G)\]
	is an isomorphism.
\end{corollary}
\begin{proof}
	By Theorem \ref{thmFJimpliesGoingDown} we may assume that $G$ is either finite or a semi-direct product $K\rtimes \Z$ for a finite group $K$. In the first case, the corollary follows from Theorem \ref{lem:virtualEqFiniteGroup}. 
	We now treat the second case. 
	By Lemma \ref{lem:idempotentcomplete}, we have an equivalence of categories
		\[\Idem(\mathcal F^p(R[H])\otimes \mathcal F^p(R[H]))\simeq \mathcal F^p(R[H]\otimes R[H])\simeq \mathcal F^p(R[H\times H]).\]
	This category is regular by Lemma \ref{lem:regular} and thus $\mathcal F^p(R[H])\otimes \mathcal F^p(R[H])$ is regular by \cite[Lemma 6.4]{bartels2022vanishing}. By induction, $\mathcal F^p(R[H])^{\otimes n}$ is regular for every $n\geq 1$. 
	Note that the direct sum decomposition $R[H]\cong R\oplus I_R[H]$ induces an equivalence\footnote{This equivalence is different from the \emph{virtual} equivalence $F$.} of categories $\mathcal F^p(R[H])\simeq \mathcal F^p(R)\oplus \mathcal F^p(I_R[H])$. In particular, $(\mathcal F^p(R)\oplus \mathcal F^p(I_R[H]))^{\otimes n}\simeq \mathcal F^p(R[H])^{\otimes n}$ is regular for every $n\geq 1$ as well.
	Moreover, the functors 
		\begin{align*}
			\mathcal F^p(R[H])^{\otimes Y}\rtimes K&\to \mathcal F^p(R[H])^{\otimes Y'}\rtimes K\\
			(\mathcal F^p(R)\oplus \mathcal F^p(I_R[H]))^{\otimes Y}\rtimes K&\to (\mathcal F^p(R)\oplus \mathcal F^p(I_R[H]))^{\otimes Y'}\rtimes K
		\end{align*}
	induced by inclusions of finite $K$-invariant subsets $Y\subseteq Y'\subseteq Z$ are clearly flat inclusions since they are given by taking tensor products with a fixed finitely generated projective module.
	Now the corollary follows from Lemma \ref{lem:virtually cyclic}. 
\end{proof}

%	\begin{theorem}\label{thm:virtualequivalenceimpliesiso}
%		Let $G$ be a discrete group satisfying the Farrell--Jones conjecture with coefficients. Let $Z$ be a $G$-set, and let $F\colon \mathcal A\to \mathcal B$ be a based virtual equivalence of $R$-additive categories. Then the induced map 	
%		\[K_*(\mathcal A^{\otimes Z}\rtimes G)\to K_*(\mathcal B^{\otimes Z}\rtimes G)\]
%		is an isomorphism. 
%	\end{theorem}
%	\begin{proof}
%		Combine Lemma \ref{lem:virtualEqFiniteGroup}, Lemma \ref{lem:cyclicgroups} and Theorem \ref{thmFJimpliesGoingDown}.
%	\end{proof}
	
	\begin{lemma}\label{lem:Green}
		Let $G$ be a group, let $Z$ be a $G$-set and let $\mathcal A=\bigoplus_{z\in Z}\mathcal A_z$ be an $R$-linear category with $G$-action $\alpha\colon G\curvearrowright \mathcal A$ such that for every $g\in G$ and $z\in Z$, $\alpha$ restricts to an isomorphism $\alpha_g\colon \mathcal A_z\xrightarrow{\simeq} \mathcal A_{gz}$. Then we have an equivalence of $R$-linear categories
		\[\mathcal A\rtimes G\simeq \bigoplus_{[z]\in G\backslash Z}\mathcal A_z\rtimes G_z\]
		where $G_z$ denotes the stabilizer of $z\in Z$. 
	\end{lemma}
	\begin{proof}
		By decomposing both sides as direct sums over the orbit space $G\backslash Z$, we may assume that $Z$ is transitive. Pick $z_0\in Z$. Then the natural inclusion functor 
		\[\mathcal A_{z_0}\rtimes G_{z_0}\to \mathcal A\rtimes G_{z_0}\to \mathcal A\rtimes G\]
		is fully faithful by construction. It is essentially surjective since $Z$ is transitive. 
	\end{proof}
	
	\begin{theorem}\label{thm:wreathproducts}
		Let $G$ be a group satisfying the $K$-theoretic Farrell--Jones conjecture with coefficients and let $H$ be a finite group. Assume that $R$ is regular and that the orders of $H$ and of every finite subgroup of $G$ are invertible in $R$. 
		Then we have an isomorphism 
		\[K_*(R[H\wr G])\cong \bigoplus_{[F]\in G\backslash\mathrm{FIN}(G)}K_*(I_R[H]^{\otimes F}\rtimes G_F),\]
		where $\mathrm{FIN}(G)$ denotes the set of finite subsets of $G$. 
	\end{theorem}
	\begin{proof}
%		By Example \ref{eg:finite group ring}, we have a based virtual equivalence of $R$-additive categories
%		\begin{equation}\label{eq:virtual eq for group ring}
%		\mathcal F^f(R)\oplus \mathcal F^f(I_R[H])\xrightarrow{\simeq} \mathcal F^f(R[H]) 
%		\end{equation}
%		where the base points are given by the rank one modules $R\in \mathcal F^f(R)$ and $R\in \mathcal F^f(R[H])$. 
%		With respect to this basepoint, we have an equivalence 
%		\[(\mathcal F^f(R)\oplus \mathcal F^f(I_R[H]))^{\otimes G}\simeq \bigoplus_{F\in \FIN(G)}\mathcal F^f(I_R[H])^{\otimes F}\]
%		of $R$-additive categories with $G$-action, where we use the convention $\mathcal F^f(I_R[H])^{\otimes \emptyset}\coloneqq \mathcal F^f(R)$.
		Note that there is a canonical isomorphism 
			\[R[H\wr G]\cong R[H]^{\otimes G}\rtimes G\]
		of $R$-algebras. Using Lemma \ref{lem:idempotentcomplete}, we obtain an equivalence 
			\begin{equation}\label{eq:wreath=Bernoulli}
			\mathcal F^p(R[H\wr G])\simeq \Idem(\mathcal F^p(R[H])^{\otimes G}\rtimes G).
			\end{equation}
		Equipping $\mathcal F^p(R)\oplus \mathcal F^p(I_R[H])$ with the rank one module $R\in \mathcal F^p(R)$ as a base point, we have an equivariant equivalence (c.f. the binomial formula)
		\[(\mathcal F^f(R)\oplus \mathcal F^f(I_R[H]))^{\otimes G}\simeq \bigoplus_{F\in \FIN(G)}\mathcal F^f(I_R[H])^{\otimes F}.\]
		By Lemmas \ref{lem:idempotentcomplete} and \ref{lem:Green}, this induces an equivalence 
			\begin{equation}\label{eq:distributive}
			\begin{aligned}
				&\Idem\left((\mathcal F^p(R)\oplus \mathcal F^p(I_R[H]))^{\otimes G}\rtimes G\right)\\
				 \simeq &\Idem\left(\left(\bigoplus_{F\in \FIN(G)}\mathcal F^p(I_R[H])^{\otimes F}\right)\rtimes G\right)\\
				\simeq &\Idem\left(\bigoplus_{F\in G\backslash\FIN(G)}\mathcal F^p(I_R[H])^{\otimes F}\rtimes G_F\right).
			\end{aligned}
			\end{equation}
		Combining Corollary \ref{cor:grouprings} with \eqref{eq:wreath=Bernoulli} and \eqref{eq:distributive}, we obtain the desired isomorphism 
		\[K_*(R[H\wr G])\cong \bigoplus_{[F]\in G\backslash\mathrm{FIN}(G)}K_*(I_R[H]^{\otimes F}\rtimes G_F).\]
	\end{proof}

	\begin{remark}\label{rem:wreathproducts} Theorem \ref{thm:wreathproducts} generalizes to an isomorphism
		\[K_*(A[H\wr G])\cong \bigoplus_{[F]\in G\backslash\mathrm{FIN}(G)}K_*(A\otimes I_R[H]^{\otimes F}\rtimes G_F),\]
	for any regular (noncommutative) $R$-algebra $A$ with unit such that the orders of $H$ and of every finite subgroup of $G$ are invertible in $A$. Here, without loss of generality, we may let $R$ be a suitable localization $\mathbb{Z}[S^{-1}]$ of $\mathbb{Z}$ where $S$ contains the orders of $H$ and of every finite subgroup of $G$. Indeed, under these assumptions, the proof of Corollary \ref{cor:grouprings} generalizes to show that $K_*(\id_A\otimes F^{\otimes Z}\rtimes G)$ induces an isomorphism
	\[ K_*(A\otimes (\mathcal F^p(R)\oplus \mathcal F^p(I_R[H]))^{\otimes Z}\rtimes G)\cong K_*(A\otimes \mathcal F^p(R[H])^{\otimes Z}\rtimes G).\]
The right-hand side is naturally identified as $K_*(A[H\wr G])$, and the left-hand side decomposes as in the proof of Theorem \ref{thm:wreathproducts}.
	\end{remark}

	For algebraically closed fields of characteristic zero, we obtain a much more concrete $K$-theory formula analogous to the results in \cite{XinLi}.
	\begin{corollary}\label{cor:complexgroupring}
		Let $G$ be a group satisfying the $K$-theoretic Farrell--Jones conjecture with coefficients and let $H$ be a finite group. 
		Assume that $R$ is an algebraically closed field of characteristic zero.
		Then we have an isomorphism 
		\begin{align*}
			K_*(R[H\wr G])
			&\cong \bigoplus_{[F]\in G\backslash \FIN(G)} \bigoplus_{[S] \in G_F\backslash (\{1, \ldots, n\}^F) } K_*(R[G_S])\\
			&\cong  K_*(R[G]) \oplus \bigoplus_{[C]\in \mathcal C} \bigoplus_{[X] \in N_C \backslash F(C) }  \bigoplus_{[S] \in C \backslash \{1, \ldots, n\}^{C\cdot X} } K_\ast(R[C_S]).			
		\end{align*}
		Here, $n$ denotes the number of non-trivial conjugacy classes of $H$, $\mathcal C$ denotes the set of all conjugacy classes of finite subgroups of $G$, $F(C)$  the nonempty finite subsets of $C\backslash G$ which are not of the form $\pi^{-1}(Y)$ for a finite subgroup $D\subseteq G$ with $C\subsetneq D$ and $Y\subseteq D\backslash G$ where $\pi\colon C\backslash G\to D\backslash G$ denotes the projection, $N_C=\{g\in G: gCg^{-1}=C\}$ the normalizer of $C$ in $G$, and $C_S=G_S\cap C$ the stabilizer of $S$ in $C$.
	\end{corollary}
	\begin{proof}
		Since $R$ is algebraically closed of characteristic zero and since $H$ is finite, we have a Morita-equivalence $I_R[H]\simeq R^{\oplus n}$ and therefore an equivalence of categories 
			\begin{equation}\label{eq:complexaugmentation}
				\mathcal F^p(I_R[H])\cong \mathcal F^p(R^{\oplus n}).
			\end{equation}
		We obtain isomorphisms 
		\begin{align}
			K_*(R[H\wr G])&\cong \bigoplus_{[F]\in G\backslash\mathrm{FIN}(G)}K_*(I_R[H]^{\otimes F}\rtimes G_F) \label{eq:complex1} \\ 
							&\cong \bigoplus_{[F]\in G\backslash\mathrm{FIN}(G)}K_*((R^{\oplus n})^{\otimes F}\rtimes G_F) \label{eq:complex2} \\
							&\cong \bigoplus_{[F]\in G\backslash \FIN(G)} \bigoplus_{[S] \in G_F\backslash (\{1, \ldots, n\}^F) } K_*(R[G_S]). \label{eq:complex3}
		\end{align}
	Here we have used Theorem \ref{thm:wreathproducts} in \eqref{eq:complex1}, \eqref{eq:complexaugmentation} and Lemma \ref{lem:idempotentcomplete} in \eqref{eq:complex2}, and Lemma \ref{lem:Green} in \eqref{eq:complex3}. 
	This proves the first isomorphism of the corollary. The second isomorphism can be derived using the same combinatorial arguments as in \cite[Proposition 2.4]{XinLi}.
	\end{proof}
	
	\section{Bernoulli shifts on semi-simple algebras}
	This section is devoted to the computation of $K_*((\mathcal A\otimes {M_n^{\otimes Z}})\rtimes G)$ where $G$ is a group, $\mathcal A$ is an $R$-linear category with $G$-action, $Z$ is a $G$-set and $M_n$ denotes the $R$-algebra of $n\times n$-matrices over $R$ for some $n\geq 1$. We use this to calculate the $K$-theory of $A^{\otimes Z}\rtimes G$ where $A=M_{k_0}\oplus \dotsb\oplus M_{k_n}$ is a semi-simple $R$-algebra that decomposes into a finite sum of matrix algebras. 
	This is mainly done by analysing the $A(G)$-module structure of $K_*(\mathcal A\rtimes G)$ where $A(G)$ is the \emph{Burnside ring} of $G$, given by the Grothendieck group of the monoid of isomorphism classes of finite $G$-sets. The addition in $A(G)$ is given by disjoint union and the multiplication is given by cartesian product.
	The \emph{Swan ring} $\Sw{G}$ of $G$ is the $K_0$-group of the exact category of $\Z[G]$-modules that are finitely generated and free as $\Z$-modules, equipped with the tensor product as multiplication. Explicitely, $\Sw{G}$ is generated by finitely generated $\Z$-free $\Z[G]$-modules with the relation $[X] + [Z] = [Y]$ whenever $0\to X\to Y\to Z\to 0$ is a short exact sequence. 
	There is a canonical homomorphism $A(G)\to \Sw{G}$ given by $X\mapsto \Z[X]$. 
	\begin{lemma}[{\cite[Lemma 9.1]{bartels2007coefficients}}]\label{lem:Swan}
		Let $G$ be a group and let $\mathcal A$ be an additive category with $G$-action. Then the algebraic $K$-theory $K_*(\mathcal A\rtimes G)$ is a module over $\Sw{G}$ in a natural way. In particular, $K_*(\mathcal A\rtimes G)$ is a module over $A(G)$. 
	\end{lemma}
	\begin{proof}
		The main idea is to associate to a $\Z[G]$-module $M\cong\Z^d$ the functor $\mathcal A\rtimes G \to \mathcal A\rtimes G$ which maps an object $A$ to the object $\bigoplus_{i=1}^d A\eqcolon A\otimes_\Z M$ and a morphism $\sum_{g\in G}f_g g$ to the morphism $\sum_{g\in G}(f_g \otimes_\Z l_g) g$ where $l_g\colon M\to M$ denotes the action by $g$. 
		We refer to \cite[Lemma 9.1]{bartels2007coefficients} for the the details.
	\end{proof}
			
	\begin{lemma}\label{lem:SwanActionAndUHF}
		Let $G$ be a group, let $X$ be a finite $G$-set, and let $\mathcal A$ be a $G$-$R$-linear category. Then there is an equivalence of $R$-linear categories 
		\[\Phi\colon \Idem((\mathcal A\otimes \End_R(R[X]))\rtimes G)\simeq \Idem(\mathcal A\rtimes G)\]
		such that the induced map 
		\[K_*(\mathcal A\rtimes G)\xrightarrow{(\id\otimes 1)\rtimes G} K_*((\mathcal A\otimes \End_R(R[X]))\rtimes G)\xrightarrow[\cong]{\Phi}K_*(\mathcal A\rtimes G)\]
		agrees with the action of the element $[X]\in A(G)$ from Lemma \ref{lem:Swan}. 
	\end{lemma}
	\begin{proof}
		The $\End_R(R[X])$-$R$-bimodule $R[X]$ implements an equivalence of categories
			\[F\colon \mathcal F^p(\End_R(R[X]))\to \mathcal F^p(R),\quad M\mapsto M\otimes_{\End_R(R[X])}R[X].\]
		This equivalence is not strictly equivariant, but equivariant in the sense of \cite[Definition 2.1]{BartelsLueck2010}, meaning that there are compatible natural isomorphisms $T_g\colon F\circ \alpha_g \Rightarrow \beta_g \circ F$ for every $g\in G$ where $\alpha\colon G\curvearrowright \mathcal F^p(\End_R(R[X]))$ and $\beta\colon G\curvearrowright \mathcal F^p(R)$ are the natural actions. Concretely, $T_g$ is given by 
			\[T_g \colon \alpha_g(M)\otimes_{\End_R(R[X])}R[X]\to \beta_g(M\otimes_{\End_R(R[X])}R[X]),\quad a\otimes b\mapsto a\otimes \gamma_{g^{-1}}(b),\]
		where $\gamma\colon G\curvearrowright R[X]$ denotes the natural action. 
		We obtain an equivalence of categories 
		\[\Psi\colon (\mathcal A\otimes \mathcal F^p(\End_R(R[X])))\rtimes G\xrightarrow{\simeq} (\mathcal A\otimes \mathcal F^p(R))\rtimes G\]
		given by $A\otimes M\mapsto A\otimes F(M)$ on objects and by \[\sum_{g\in G}f_g g\mapsto \sum_{g\in G}(\id_\mathcal A \otimes F)(f_g)\circ (\id_\mathcal A \otimes T_g^{-1}) g\] on morphisms. 
	
		Using Lemma \ref{lem:idempotentcomplete}, we can identify $\Idem(\Psi)$ with an equivalence 
		\[\Phi\colon \Idem((\mathcal A\otimes \End_R(R[X]))\rtimes G)\xrightarrow{\simeq} \Idem(\mathcal A\rtimes G).\]
		By spelling out the explicit construction of the $\Sw{G}$-action in \cite[Lemma 9.1]{bartels2007coefficients}, it is easy to see that $\Phi$ induces the desired isomorphism on $K$-theory.
	\end{proof}

	The following proposition is the algebraic $K$-theory analogue of \cite[Proposition 2.1]{KN}.
	
	\begin{proposition}\label{prop:kkG}
		Let $G$ be a finite group, let $Z$ be a finite $G$-set and let $n\in \N$ be an integer. Denote by $[n^Z]\in A(G)$ the element represented by the finite $G$-set $\{1,\dotsc,n\}^Z$. Then there are elements $\alpha,\beta\in A(G)$ and integers $l,r\geq 1$ such that 
		\begin{enumerate}
			\item $[n^Z]^l=n\cdot \alpha\in A(G)$,
			\item $[n^Z]\cdot \beta=n^r\in A(G)$. 
		\end{enumerate}
		Furthermore, $\alpha$ can be represented by a $G$-set. 
	\end{proposition}
	\begin{proof}
		Denote by $\lambda\colon A(G)\to \End_\Q(A(G)\otimes \Q)$ the representation given by left multiplication. 
		Let $(H_i)_{1\leq i\leq k}$ be a collection of subgroups of $G$ containing each conjugacy class exactly once. 
		Then $([G/H_i])_{1\leq j\leq k}$ is a basis for $A(G)$ so that any element $x\in A(G)$ can uniquely be written as $x=\sum_{1\leq i\leq k}n_i [G/H_i]$ with $n_i\in \Z$. 
		By reordering the $H_i$ according to the maximal length of any chain of proper subgroups contained in $H_i$, we may furthermore assume that $i\leq j$ whenever $H_i$ is subconjugate to $H_j$.
		We make the following observations:
		\begin{enumerate}
			\item The coefficients $n_i$ of an element $x=\sum_i n_i [G/H_i]$ may be recovered as the last column of the matrix representing $\lambda(x)$ with respect to the basis $([G/H_i])_{1\leq i \leq k}$ of $A(G)_\Q$. In particular, $\lambda$ is injective. 
			\item Let $X$ be any finite $G$-set and $1\leq i\leq k$. Then the stabilizers of $X\times G/H_i$ are all subconjugate to $H_i$. In particular, the matrix representing $\lambda([X])$ with respect to $([G/H_i])_{1\leq i\leq k}$ is an upper triangular matrix. 
			\item For any $1\leq i\leq k$, the $G$-set $\{1,\dotsc,n\}^Z\times G/H_i$ has $n^{|Z/H_i|}$-many orbits of type $G/H_i$. 
		\end{enumerate}
		The second and the third condition imply 
		\[\prod_{i=1}^k(\lambda([n^Z])-n^{|Z/H_i|})=0\in \End_\Q(A(G)\otimes \Q),\]
		while the first condition implies 
		\[\prod_{i=1}^k([n^Z]-n^{|Z/H_i|})=0\in A(G).\]
		Now the proposition follows by letting $\alpha$ and $\beta$ be suitable polynomials in $[n^Z]$. The equality $[n^Z]^l=n\cdot \alpha$ implies that all the coefficients of $\alpha$ are positive. In particular $\alpha$ can be represented by a $G$-set. 
	\end{proof}

	\begin{lemma}\label{lem:ssa}
		Let $\mathcal A$ be an $R$-linear category. Then the canonical inclusion 
		\[\id\otimes 1\colon \mathcal A\otimes M_n^{\otimes \N}\to \mathcal A\otimes M_n^{\otimes \N}\otimes M_n^{\otimes \N}\]
		induces an isomorphism on $K$-theory. 
	\end{lemma}
	\begin{proof}
		The proof is an algebraic version of an \emph{Elliott intertwining} which is used in \cite[Examples 1.14(i)]{TW} to prove that UHF-algebras of infinite type are strongly self-absorbing.  
		Let $l\geq k\geq 1$ be integers. 
		For an embedding $\sigma\colon \{1,\dotsc,k\}\to \{1,\dotsc,l\}$, we call the map 
			\[M_n^{\otimes k}\to M_n^{\otimes l},\quad a_1\otimes \dotsb\otimes a_k\mapsto b_1\otimes \dotsb\otimes b_l,\]
			\[b_j=\begin{cases}a_i,&\sigma(i)=j\\ 1,&j\notin \sigma(\{1,\dotsc,k\})		\end{cases}\]
		a \emph{standard embedding}.
		Note that any two embeddings $M_n^{\otimes k}\to M_n^{\otimes l}$ induced by different embeddings $\{1,\dotsc,k\}\to \{1,\dotsc,l\}$ are conjugate via permutation matrices. From this it follows that any two 'standard' embeddings $\mathcal A\otimes M_n^{\otimes k}\to \mathcal A\otimes M_n^{\otimes l}$ are equivalent as functors. Now consider the diagram 
		\[\xymatrix{
			\mathcal A\otimes M_n \ar[r]\ar[d]		&\mathcal A\otimes M_n^{\otimes 2}\ar[r]\ar[d]		&\mathcal A\otimes M_n^{\otimes 4}\ar[r]\ar[d]	&\dotsb \\
			\mathcal A\otimes M_n\otimes M_n\ar[r]\ar[ur]^\cong	&\mathcal A\otimes M_n^{\otimes 2}\otimes M_n^{\otimes 2}\ar[r]\ar[ur]^\cong &\mathcal A\otimes M_n^{\otimes 4}\otimes M_n^{\otimes 4}\ar[r]\ar[ur]^\cong &\dotsb,
		}\]
		where the upper horizontal arrows and the vertical arrows are induced by $\id\otimes 1$, the lower horizontal arrows are induced by $(\id\otimes 1)\otimes (\id\otimes 1)$, and the diagonal arrows are the canonical isomorphisms. 
		This diagram commutes up to equivalence of functors. 
		Since $K$-theory preserves filtered colimits (see \cite[Corollary 7.2]{Nonconnective}) and is invariant under equivalence of functors, the colimit of the vertical functors induces an isomorphism on $K$-theory. This proves the lemma. 
	\end{proof}
	
	\begin{lemma}\label{lem:UHF Z to ZN}
		Let $G$ be a finite group, let $n\geq 1$ be an integer and let $Z$ be an infinite $G$-set. Then for any additive category $\mathcal A$, the canonical inclusion
		\[\id\otimes 1\colon M_n^{\otimes Z}\to M_n^{\otimes Z}\otimes M_n^{\otimes \N}\]
		induces an isomorphism 
		\[K_*((\mathcal A\otimes M_n^{\otimes Z})\rtimes G)\xrightarrow{\cong}K_*((\mathcal A\otimes M_n^{\otimes Z}\otimes M_n^{\otimes \N})\rtimes G).\]
		
	\end{lemma}
	\begin{proof}
		Since $G$ is finite and $Z$ is infinite, it contains infinitely many orbits of the same type. Without loss of generality we may thus assume that $Z$ is of the form $Z=\sqcup_\N G/H$ for some subgroup $H\subseteq G$. It follows from the first part of Proposition \ref{prop:kkG} that there is a $G$-set $X$, an integer $l\geq 1$ and an equivariant isomorphism 
		\[(M_n^{\otimes G/H})^{\otimes l}\cong M_n\otimes \End_R(R[X]).\]
		In particular, we have an equivariant isomorphism 
		\[M_n^{\otimes Z}\cong M_n^{\otimes \N}\otimes \End_R(R[X])^{\otimes \N}.\]
		Observe moreover that there is a canonical isomorphism 
		\[(M_n^{\otimes \N}\otimes \End_R(R[X])^{\otimes \N})\rtimes G \cong M_n^{\otimes \N}\otimes (\End_R(R[X])^{\otimes \N}\rtimes G).\]
		of $R$-linear categories. 
		Now the lemma follows from Lemma \ref{lem:ssa}. 
	\end{proof}
	
	\begin{lemma}\label{lem: UHF N to ZN}
		Let $G$ be a finite group, let $n\geq 1$ be an integer and let $Z$ be a $G$-set. Then for any additive category $\mathcal A$, the canonical inclusion
		\[M_n^{\otimes \N}\to M_n^{\otimes \N}\otimes M_n^{\otimes Z}\]
		induces an isomorphism 
		\begin{equation}\label{eq:N to N Z}
		K_*((\mathcal A\otimes M_n^{\otimes \N})\rtimes G)\xrightarrow{\sim}K_*((\mathcal A\otimes M_n^{\otimes \N}\otimes M_n^{\otimes Z})\rtimes G).
		\end{equation}
	\end{lemma}
	\begin{proof}
		Since $K$-theory preserves filtered colimits (see \cite[Corollary 7.2]{Nonconnective}), we may assume that $Z$ is finite (the general case follows by taking the colimit over all finite subsets of $Z$).
		By Lemma \ref{lem:SwanActionAndUHF}, we can identify the map \eqref{eq:N to N Z} with the map 
		\[[n^Z]\cdot\colon K_*((\mathcal A\otimes M_n^{\otimes \N})\rtimes G)\to K_*((\mathcal A\otimes M_n^{\otimes \N})\rtimes G),\]
		where $[n^Z]\in A(G)$ acts as in Lemma \ref{lem:Swan}.
		By applying Lemma \ref{lem:SwanActionAndUHF} repeatedly to the trivial $G$-set $\{1,\dotsc,n\}$, this map can be identified with the map 
		\[[n^Z]\cdot \colon K_*(\mathcal A\rtimes G)[1/n]\to K_*(\mathcal A\rtimes G)[1/n].\]
		But this map is invertible by the second part of Proposition \ref{prop:kkG}. 
	\end{proof}

	\begin{theorem}\label{thm:UHF}
		Let $G$ be a group satisfying the Farrell--Jones conjecture with coefficients. Assume that the orders of all finite subgroups of $G$ are invertible in $R$. 
		Let $\mathcal A$ be a regular $G$-$R$-linear category, let $Z$ be an infinite $G$-set and let $n\geq 1$ be an integer. Then the canonical maps 
			\[K_*((\mathcal A\otimes M_n^{\otimes Z})\rtimes G)\to K_*((\mathcal A\otimes M_n^{\otimes Z}\otimes M_n^{\otimes \N})\rtimes G)\leftarrow K_*((\mathcal A\otimes M_n^{\otimes \N})\rtimes G)\]
		are isomorphisms. In particular, we have 
		\[K_*((\mathcal A\otimes M_n^{\otimes Z})\rtimes G)\cong K_*(\mathcal A\rtimes G)[1/n].\]
	\end{theorem}
	\begin{proof}
		By Theorem \ref{thmFJimpliesGoingDown} we may assume that $G$ is finite-by cyclic. If $G$ itself is finite, then the theorem follows from Lemmas \ref{lem:UHF Z to ZN} and \ref{lem: UHF N to ZN}. If $G$ is not finite, then $G$ is a semidirect product $H\rtimes_\varphi\Z$ for some finite group $H$. 
		As in the proof of Lemma \ref{lem:virtually cyclic}, we get canonical isomorphisms
		 	\begin{equation}\label{eq:semidirectproduct2}
		 	\begin{aligned}
		 	(\mathcal A\otimes M_n^{\otimes Z})\rtimes G&\cong ((\mathcal A\otimes M_n^{\otimes Z})\rtimes H)\rtimes_\Phi \Z\\
		 	(\mathcal A\otimes M_n^{\otimes Z}\otimes M_n^{\otimes \N})\rtimes G&\cong ((\mathcal A\otimes M_n^{\otimes Z}\otimes M_n^{\otimes \N})\rtimes H)\rtimes_\Phi \Z\\
			(\mathcal A\otimes M_n^{\otimes \N})\rtimes G&\cong ((\mathcal A\otimes M_n^{\otimes \N})\rtimes H)\rtimes_\Phi \Z
		 	\end{aligned}
		 	\end{equation}
		 Note that $(\mathcal A\otimes M_n^{\otimes Z})\rtimes H$ is the increasing union of the subcategories $(\mathcal A\otimes M_n^{\otimes Y})\rtimes H$ where $Y$ ranges over all finite $H$-invariant subsets of $Z$. Each of these subcategories is regular by Lemma \ref{lem:regular}\footnote{Here we also use that regularity passes to matrix amplifications, which can be deduced for instance from \cite[Lemma 6.4]{bartels2022vanishing}.} and the inclusions $(\mathcal A\otimes M_n^{\otimes Y})\rtimes H\to (\mathcal A\otimes M_n^{\otimes Y'})\rtimes H$ for finite $H$-invariant subsets $Y\subseteq Y'\subseteq Z$ are easily seen to be flat.
		 Thus their union is regular coherent by \cite[Lemma 11.2]{bartels2022vanishing}. Combined with \cite[Theorem 10.1]{bartels2022vanishing}, this even shows that $((\mathcal A\otimes M_n^{\otimes Z})\rtimes H)[\Z^m]$ regular coherent for any $m\geq 0$. The same argument shows that $((\mathcal A\otimes M_n^{\otimes Z}\otimes M_n^{\otimes \N})\rtimes H)[\Z^m]$ and $((\mathcal A\otimes M_n^{\otimes \N})\rtimes H)[\Z^m]$ are regular coherent for any $m\geq 1$. 
		 In view of the decompositions \eqref{eq:semidirectproduct2}, \cite[Theorem 7.8]{bartels2022vanishing} now reduces the problem to checking that the maps 
		 	\[K_*((\mathcal A\otimes M_n^{\otimes Z})\rtimes H)\to K_*((\mathcal A\otimes M_n^{\otimes Z}\otimes M_n^{\otimes \N})\rtimes H)\leftarrow K_*((\mathcal A\otimes M_n^{\otimes \N})\rtimes H)\]
		 are isomorphisms. But this case was already covered above. 
	\end{proof}
	As an application, we obtain $K$-theory formulas for Bernoulli-shifts on many semi-simple $R$-algebras:
	\begin{theorem}\label{thm-semisimple}
		Let $G$ be a group satisfying the Farrell--Jones conjecture with coefficients. Assume that $R$ is regular and that the orders of all finite subgroups of $G$ are invertible in $R$. Let $Z$ be an infinite $G$-set and let $A$ be an $R$-algebra of the form $A=M_{n_0}(R)\oplus \dotsb\oplus M_{n_k}(R)$. Write $n=\mathrm{gcd}(n_0,\dotsc,n_k)$. Then we have an isomorphism 
			\[K_*(A^{\otimes Z}\rtimes G)\cong \bigoplus_{[F]\in G\backslash \FIN(Z)}\bigoplus_{[S]\in G_F\backslash \{1,\dotsc,k\}^F}K_*(R[G_S])[1/n].\]
	\end{theorem}
	\begin{proof}
%		It is not hard to see that up to equivalence, $R$-additive functors 
%			\[F\colon \mathcal F^f(M_{n_0}(R))\oplus \dotsb \oplus \mathcal F^f(M_{n_k}(R))\to\mathcal F^f(M_{m_0}(R))\oplus \dotsb \oplus \mathcal F^f(M_{m_k}(R))\]\todo{There is this subtlety with free vs projective again}
%		are classified by $k\times k$-matrices $X$ with non-negative entries (
%		and that such a functor is induced by a unital algebra homomorphism 
%			\[M_{n_0}(R)\oplus \dotsb \oplus M_{n_k}(R)\to M_{m_0}(R)\oplus \dotsb \oplus M_{m_k}(R)\]
%		if and only if the representing matrix $X$ satisfies 
%			\[X\begin{pmatrix}
%			n_0\\
%			\vdots\\
%			n_k
%			\end{pmatrix}=\begin{pmatrix}
%			m_0\\
%			\vdots\\
%			m_k
%			\end{pmatrix}.\]		
%		Our assumptions together with \cite[Corollary 3.4]{CEKN} imply that we can find a unital 
		Note hat every $(k+1)\times (k+1)$-matrix $X$ with non-negative integer entries defines a functor 
		$F_X\colon \mathcal F^f(R)^{\oplus k+1}\to \mathcal F^f(R)^{\oplus k+1}$ by sending the generator $R\in \mathcal F^f(R)$ of the $j$-th direct summand to the module $R^{\oplus X_{ij}}\in \mathcal F^f(R)$ in the $i$-th direct summand. 
		Now identify $\mathcal F^p(R)\simeq \mathcal F^p(M_m)$ using the $M_m$-$R$-bimodule $R^{\oplus m}$ for $m=n,n_0,\dotsc,n_k$. 
		Using this identification, the functor 
			\begin{equation}\label{eq:IdemFX}
			\Idem(F_X)\colon \mathcal F^p(M_n)^{\oplus k+1}\to \mathcal F^p(M_{n_0})\oplus \dotsb \oplus\mathcal F^p(M_{n_k})
			\end{equation}
		is based with respect to the rank one modules $M_n\oplus \dotsb\oplus M_n$ and $M_{n_0}\oplus \dotsb\oplus M_{n_k}$ if and only if the representing matrix $X$ satisfies 
			\[X\begin{pmatrix}
			n\\
			\vdots\\
			n
			\end{pmatrix}=\begin{pmatrix}n_0\\
			\vdots\\
			n_k
			\end{pmatrix}.\]
		By \cite[Corollary 3.4]{CEKN} we can find such a matrix satisfying in addition $X\in \SL(k+1,\Z)$. 
		By writing the inverse of $X$ as a difference of two matrices with non-negative integer entries, we conclude that the functor \eqref{eq:IdemFX} is a based virtual equivalence. 
		Combining Lemmas \ref{lem:virtualEqFiniteGroup} and \ref{lem:virtually cyclic} with Theorem \ref{thmFJimpliesGoingDown}, we conclude that the induced map 
			\begin{equation}\label{semisimple1}
				K_*((M_n^{\oplus k+1})^{\otimes Z}\rtimes G)\to K_*((M_{n_0}\oplus \dotsb\oplus M_{n_k})^{\otimes Z}\rtimes G)
			\end{equation}
		is an isomorphism. Using the isomorphism $M_n^{\oplus k+1}\cong M_n\otimes R^{\oplus k+1}$, Theorem \ref{thm:UHF} moreover provides us with an isomorphism 
			\begin{equation}\label{semisimple2}
				K_*((M_n^{\oplus k+1})^{\otimes Z}\rtimes G)\cong K_*((R^{\oplus k+1})^{\otimes Z}\rtimes G)[1/n].
			\end{equation}
		Now note that the virtual equivalence 
			\[F_X\colon \mathcal F^f(R)\oplus \dotsb\oplus \mathcal F^f(R)\to \mathcal F^f(R)\oplus \dotsb\oplus \mathcal F^f(R)\]
		induced by the invertible matrix 
			\[X=\begin{pmatrix}
			1 & &1\\
			&\ddots&\vdots\\
			&&1
			\end{pmatrix}\in \SL(k+1,\Z)\]
		is based with respect to the base points $0\oplus \dotsb \oplus 0\oplus R$ and $R\oplus \dotsb\oplus R$.
		Comparing the infinite tensor products with respect to these different base points and applying the combination of Theorem \ref{lem:virtualEqFiniteGroup} and Lemma \ref{lem:virtually cyclic} with Theorem \ref{thmFJimpliesGoingDown}, we obtain an isomorphism 
			\begin{equation}
				K_*((R^{\oplus k+1})^{\otimes Z}\rtimes G)\cong K_*\left(\left(\bigoplus_{F\in \FIN(Z)}(R^{\oplus k})^{\otimes F}\right)\rtimes G\right)
			\end{equation}
		In combination with Lemma \ref{lem:Green}, we get an isomorphism 
			\begin{align*}	
				K_*((R^{\oplus k+1})^{\otimes Z}\rtimes G)&\cong \bigoplus_{[F]\in G\backslash \FIN(Z)}K_*((R^{\oplus k})^{\otimes F}\rtimes G_F)\\
				&\cong \bigoplus_{[F]\in G\backslash \FIN(Z)}K_*(R^{\oplus \{1,\dotsc,k\}^F}\rtimes G_F)\\
				&\cong \bigoplus_{[F]\in G\backslash \FIN(Z)}\bigoplus_{[S]\in G_F\backslash \{1,\dotsc,k\}^F}K_*(R[G_S]),
			\end{align*}
		Now the theorem follows from combining this isomorphism with \eqref{semisimple1} and \eqref{semisimple2}. 
	\end{proof}
	\bibliography{Refs}
	\bibliographystyle{alpha}

\end{document}